\documentclass[12pt,final]{article}

\usepackage{amsmath}
\usepackage{amsthm}
\usepackage{amssymb}
\usepackage{amsfonts}
\usepackage{latexsym}               
\usepackage{amsgen}
\usepackage[mathscr]{eucal}
\usepackage{mathrsfs}
 \usepackage{enumerate}

\usepackage{mathptmx} 
\usepackage{fancybox}
\usepackage{color}
\usepackage{graphicx}

\usepackage{showkeys}

\newtheorem{lem}{Lemma}[section]
\newtheorem{thm}[lem]{Theorem}

\newcommand{\bqn}{\begin{equation}}
\newcommand{\eqn}{\end{equation}}
\newcommand{\beqx}{\begin{equation*}}
\newcommand{\eeqx}{\end{equation*}}
\newcommand{\barr}{\begin{array}}
\newcommand{\earr}{\end{array}}
\newcommand{\beqn}{\begin{eqnarray}}
\newcommand{\eeqn}{\end{eqnarray}}
\newcommand{\beqnx}{\begin{eqnarray*}}
\newcommand{\eeqnx}{\end{eqnarray*}}
\newcommand{\bmt}{\begin{multline}}
\newcommand{\emt}{\end{multline}}

\numberwithin{equation}{section}

\newcommand{\pa}{\partial} 
\newcommand{\bl}{^\bullet}

\def\det{\mathop{\rm det}\nolimits}

\newcommand{\sA}{\mathscr{A}}

\newcommand{\sF}{\mathscr{F}}
\newcommand{\sK}{\mathscr{K}}
\newcommand{\sL}{\mathscr{L}}

\newcommand{\sN}{\mathscr{N}}
\newcommand{\sO}{\mathscr{O}}

\newcommand{\sW}{\mathscr{W}}

\newcommand{\bbn}{{\mathbb N}}

\newcommand{\real}{\mathbb R}
\newcommand{\al}{\alpha}

\newcommand{\ga}{{\gamma}}

\newcommand{\ve}{\varepsilon}

\newcommand{\la}{\lambda}
\newcommand{\La}{\Lambda}

\newcommand{\ro}{\rho}

\newcommand{\er}{\eqref}
\newcommand{\lb}{\label}
\newcommand{\qu}{\quad}

\title{Steady States of Gas Ionization with Secondary Emission} 

\author{%
{\large\sc Walter A. Strauss${}^1$}
{\normalsize and}
{\large\sc Masahiro Suzuki${}^2$}
}

\date{%
\normalsize
${}^1$%
Department of Mathematics and Lefschetz Center for Dynamical Systems, 
Brown University, 
\\
Providence, RI 02912, USA
\\ [7pt]
${}^2$%
Department of Computer Science and Engineering, 
Nagoya Institute of Technology,
\\
Gokiso-cho, Showa-ku, Nagoya, 466-8555, Japan
}

\begin{document}

\maketitle


\begin{abstract}
We consider the steady states of a gas between two parallel plates that is 
ionized by a strong electric field so as to create a plasma.  
There can be a cascade of electrons due both to the electrons colliding with the gas molecules and 
to the ions colliding with the cathode (secondary emission).   We use 
global bifurcation theory to prove that there is a one-parameter family $\sK$ of such steady 
states with the following property.  The curve $\sK$ begins at the sparking 
voltage and either the particle density becomes unbounded or $\sK$ 
ends at an anti-sparking voltage. These critical voltages are characterized explicitly.  
\end{abstract}

\begin{description}
\item[{\it Keywords:}]
ionization; gas discharge; secondary emission; 
sparking voltage; 
global bifurcation; 
plasma

\item[{\it 2010 Mathematics Subject Classification:}]
35A01; 
35M12; 
70K50; 
76X05; 
82D10. 
\end{description}


\newpage

\section{Introduction}

This paper is concerned with a model for the ionization of a gas such as air 
due to a strong applied electric field.   For instance, 
the strong electric field may be created when a capacitor discharges into a gap between electrodes.  
The high voltage thereby creates a plasma, which may possess very hot or bright electrical arcs.  
A century ago Townsend experimented with a pair of parallel plates 
to which he applied a strong voltage that produced cascades of free electrons and ions.  
This phenomenon is called the Townsend discharge or avalanche. 

Such an avalanche primarily occurs due to free electrons colliding with gas molecules,
thus liberating other electrons.  This is called the $\al$-mechanism.  
Another important contribution to an avalanche may be due to the impact of ions with the cathode, 
which then emits additional electrons.  This is called the secondary emission or the $\ga$-mechanism.  
In this paper we discuss a model that takes account of both mechanisms.  

The model is as follows. 
Let $I=(0,L)$ be the distance between the planar parallel plates. 
 Let us put the anode at $x=0$ and the cathode at $x=L$.  
Let $\ro_i$ be the density of positive ions, $\ro_e$ the density of electrons, and $-\Phi$  
the electrostatic potential.  
Let $u_i$ and $u_e$ be the ion and electron velocities.  
Then the equations within the region $I$ are as follows.  
\begin{subequations}\lb{Mmodel}   
\begin{gather}
\partial_t\ro_i+\partial_x(\ro_iu_i) = 
a \exp\left(-b{|\partial_x \Phi|^{-1}} \right) \ro_e\left|v_e\right| ,
\lb{DLMi} \\
\partial_t\ro_e+\partial_x(\ro_eu_e) = 
a \exp\left(-b{|\partial_x \Phi|^{-1}} \right) \ro_e\left|v_e\right| ,
\lb{DLMe} \\
\partial_x^2 \Phi=\ro_i-\ro_e,
\lb{DLMp} \\
u_i:=k_i {\partial_x \Phi}, \ \ 
u_e:=v_e-k_e \partial_x \rho_e/\rho_e, \ \ 
v_e:=-k_e {\partial_x \Phi}, 
\lb{DLMu}
\end{gather}
\end{subequations}
Here $k_i$, $k_e$, $a$, and $b$ are positive constants. 
The constitutive velocity relations \er{DLMu} are due to the ions 
being much heavier than the electrons. 
The right sides of \er{DLMi} and \er{DLMe} come from the $\al$-mechanism.  
They express the number of ion--electron pairs generated per unit volume 
by the impacts of the electrons.  
Specifically, the coefficient $\al=a \exp\left(-b{|\partial_x \Phi|^{-1}} \right) $ is 
the first Townsend ionization coefficient. 

The boundary conditions at the anode $x=0$ are $\rho_i=\rho_e=\Phi=0$, 
due to the assumption that the anode is a perfect conductor, so that the 
electrons are absorbed by the anode and the ions are repelled from the anode. 
We denote the voltage at the cathode $x=L$ by $V_c>0$.  
The secondary emission at the cathode (or $\ga$-mechanism) 
is expressed by 
\bqn   \label{gamma}
\rho_eu_e=-\ga\rho_iu_i   \eqn
where $\ga > 0$ is average number of electrons ejected from the cathode by an ion impact. 

In this paper we consider the steady state problem, where the unknowns do not depend on time, 
even though the individual particles can move rapidly.  
First of all, there are the completely trivial solutions $\rho_i\equiv0, \rho_e\equiv0$, 
$\Phi(x)=\frac{V_c}L x$, where $V_c$ is an arbitrary constant.  
Avalanche does not occur unless the electric field is strong enough.  
In our model 
the ionization coefficient $a$  or the secondary emission coefficient $\ga$ must be large enough, 
depending on $b$ and $L$, in order to reach this threshold.  
Then the critical threshold value of the voltage is called the {\it sparking voltage} $V_c^\dagger$.  
Assuming that the sparking voltage does exist, 
we prove that there are many other steady solutions, in fact a whole global curve of them, 
for most  choices of the parameters $(a,b,\gamma)$. 

\begin{thm} \label{mainthm0}
Assume that the sparking voltage $V_c^\dagger$ exists. 
For almost every $(a,b,\gamma)$, 
there exists a unique continuous one-parameter family $\mathcal{K}$ (that is, a curve) 
of steady solutions of the system of equations 
together with the boundary conditions written above with the following properties.  
Both densities are positive, $\rho_i\in C^1, \ \rho_e\in C^2, \ \Phi\in C^3$, 
the curve begins at the trivial solution with voltage $V_c^\dagger$ 
and ``ends'' with one of the following three alternatives: 

Either the density $|\rho_i| + |\rho_e|$ becomes unbounded along $\sK$, 

Or the potential $\Phi$ becomes unbounded along $\sK$, 

Or the curve ends at a different trivial solution with some voltage $V_c^\ddagger > V_c^\dagger$.
\end{thm}

The sparking voltage $V_c^\dagger$ is the smallest positive root of a certain elementary function $D(\cdot)$,
which we call the {\it sparking function}.  
 We say that the  sparking voltage  {\it exists} for a given parameter triple $(a,b,\gamma)$ 
if $D$ has a positive root for any triple in a neighborhood of it.  
We call $V_c^\ddagger$ the {\it anti-sparking} voltage; it is a larger root of  $D(\cdot)$.  
The explicit {\it sparking function} $D$ is defined as follows.  
For brevity we first denote 
\begin{equation}  \label{gh} 
\la = \frac{V_c}L, \quad h(\la) = a \la e^{-b/\la}, \quad  g(\la L) = h(\la) - {\la^2}/{4}.   \end{equation}
Then let  $\mu=L\sqrt{-g(V_c)}$ and 
\begin{equation}  \label{spfu} 
D(V_c)= \frac{1}{2}  \left(e^{\mu}  +  e^{-\mu}\right)
+   \frac{V_c}{4 {\mu} }  \left(e^{\mu}  -  e^{-\mu} \right)
-   \frac{\gamma}{1+\gamma}  e^{\frac{V_c}{2}}.
\end{equation} 
Note that, even if $g(V_c)$ is positive, $D(V_c)$ is real.  
In case $g(V_c)$ vanishes, $D(V_c)$ is defined as the limit $\lim_{g(V_c)\to 0} D(V_c)$.  
Thus $D\in C((0,\infty);\mathbb R)$.  
Depending on $\ga, a, b$ and $L$, the sparking function $D$ may have no root, one root or several roots. 
If $D$ has a root,  the {\it sparking voltage} $V_c^\dagger>0$ is defined as the smallest one: 
\begin{equation}\lb{bp1}
V_c^\dagger := \inf \{{V_c>0}\,;\,D(V_c)=0\}.
\end{equation}
Sufficient conditions for $D$ to have one or more roots, or none, are given in  Appendix A.  

We prove Theorem \ref{mainthm0} by a local, and then a global, bifurcation argument.  
In Section 2 we set up the notation used in the analysis. 
In Section 3 we apply the well-known local bifurcation theorem. 
In particular, we prove that the nullspace and the range  of the linearized operator 
around any trivial solution is determined by the function $D$.  
A transversality condition is required in order to guarantee the local bifurcation.   
We prove in Lemma \ref{trans1} that this condition is valid for almost every $(a,b,\gamma)$.  
Then in Section 4 we apply a global bifurcation theorem 
to construct a {\it global} curve $\sK$ of steady solutions $(\rho_i,\rho_e, \Phi)$.  
The general properties of this global curve are given in Theorem \ref{Global1}. 
    The curve may include mathematical solutions with positive densities
as well as solutions with negative ``densities''.  
In Section 5 we restrict our attention to positive densities.  
Further analysis of the possible ways that the curve may ``terminate'' is then  provided.  
The main conclusion (as in Theorem \ref{mainthm0}) is given in Theorem  \ref{mainthm}.  
In case the voltage becomes unbounded, it is proven in Section 5 that the densities tend to zero. 

Appendix A is devoted to the sparking function \eqref{spfu}.  
It is shown that there is a sparking voltage 
{if either $a$ or $\gamma$ is large enough.}
In Appendix B we discuss the location of the sparking voltage \eqref{bp1}.

\section{History and Notation} 

We now briefly summarize the history of the model.  
Many models have been proposed to describe this phenomenon \cite{AB,DL1,DW,DT,Ku1,Ku2,LRE,Mo1}.  
In 1985 Morrow \cite{Mo1} was perhaps the first to provide a model of its detailed mechanism in 
terms of particle densities. 
The model  consists of continuity equations for the electrons and ions coupled to the Poisson equation 
for the electrostatic potential.  
For simplicity in this paper we consider only electrons and positive ions 
and we focus on the $\ga$ and $\al$ mechanisms.  
Various other mechanisms can occur, such as 'attachment' and 'recombination' as mentioned 
in Morrow's paper, which have a much smaller effect on the ionization. 
  
The interesting article \cite{DL1} of Degond and Lucquin-Desreux derives the model directly 
from the general Euler-Maxwell system by scaling assumptions, 
in particular by assuming a very small mass ratio between the electrons and ions. 
In an appropriate limit the Morrow model is obtained 
at the end of their paper in equations (160) and (163), which we have specialized to 
assume constant temperature and no neutral particles. 

Suzuki and Tani in \cite{ST1} gave the first mathematical analysis of the Morrow model.
Typical shapes of the cathode and anode 
in physical and numerical experiments are a sphere or a plate.
Therefore they proved the time-local solvability of an initial boundary value problem over 
domains with a pair of boundaries that are plates or spheres.
In another paper \cite{ST2} they did a deeper analysis of problem \er{Mmodel}, 
proving that there exists a certain threshold of voltage  
at which the trivial solution transitions from stable to unstable. 
This fact means that gas discharge can occur and continue 
for a voltage greater than the threshold.

In \cite{SS1} we considered the Morrow model with the $\al$-mechanism but without the 
$\ga$-mechanism.  The boundary condition \eqref{gamma} was replaced by 
the condition that $\rho_e=0$ at the cathode, which means that the electrons are simply repelled
by the cathode.  
For that simpler model the sparking voltage $V_c^\dagger$ is the smallest root of the function $g$ 
and the anti-sparking voltage  $V_c^\ddagger$ is the other root if it exists.  
We proved similarly that there is a global curve of steady solutions that starts at  $V_c^\dagger$ 
and either goes to infinity or is a half-loop that goes to  $V_c^\ddagger$.  
In that case we eliminated the alternative that the voltage may be unbounded.


Now we describe some notation that we use in the rest of the paper.  
For mathematical convenience  we rewrite the problem \er{Mmodel}
in terms of the new unknown function
\[
R_e:=\ro_e e^{\frac{V_c}{2L}x} .  
\] 
We decompose the electrostatic potential as
\[
 \Phi=V+\frac{V_c}{L}x. 
\]
Thus $\partial_x^2 V=\rho_i-e^{-\frac{V_c}{2L}x}R_e$ with the boundary conditions $V(0)=V(L)=0$.
As a result, 
from \eqref{Mmodel} we have the following system for stationary solutions: 
\begin{subequations}\lb{r0}
\begin{gather}
k_i\partial_x\left\{\left(\partial_x V+\frac{V_c}{L}\right)\rho_i\right\}
=k_eh\left(\partial_x V+\frac{V_c}{L}\right)e^{-\frac{V_c}{2L}x}R_e,
\lb{re1} \\
-k_e\partial_x^2 R_e -k_eg(V_c)R_e=k_ef_e[V_c,R_e,V],
\lb{re2} \\
\partial_x^2 V=\rho_i-e^{-\frac{V_c}{2L}x}R_e
\lb{re3} 
\end{gather}
with the boundary conditions
\begin{gather}
\rho_i(0)=R_e(0)=V(0)=V(L)=0,
\lb{rb1} \\
\partial_xR_e(L)+\left(\partial_x V(L)+\frac{V_c}{2L}\right)R_e(L)=
\gamma\frac{k_i}{k_e}e^{\frac{V_c}{2L}L}
\left(\partial_x V(L)+\frac{V_c}{L}\right)\ro_i(L), 
\lb{rb2}
\end{gather}
\end{subequations}
where the nonlinear term $f_e =  f_e[V_c,R_e,V]$ is defined as
\begin{align*}
f_e=(\partial_xV)\partial_x R_e 
-\frac{V_c}{2L}R_e\partial_xV+R_e\partial_x^2V
-\left[
h\left(\frac{V_c}{L}\right)-
h\left(\partial_xV+\frac{V_c}{L}\right)
\right]R_e.
\end{align*}

{
It is convenient to draw the graph of $g(V_c)$, which of course depends  
on the physical parameters $a$, $b$, and $L$.  
The function $g$ has at most one local maximum in $(0,\infty)$.  
\begin{figure}[htbp] \begin{minipage}{0.5\hsize}   
{\unitlength 0.1in
\begin{picture}( 29.1500, 19.9000)( -0.1100,-22.0000)
\put(1.8900,-15.2100){\makebox(0,0)[rt]{$O$}}%
\put(28.8300,-16.8600){\makebox(0,0){$V_c$}}%
%
{{%
\special{pn 13}%
\special{pa 288 2200}%
\special{pa 288 210}%
\special{fp}%
\special{sh 1}%
\special{pa 288 210}%
\special{pa 268 278}%
\special{pa 288 264}%
\special{pa 308 278}%
\special{pa 288 210}%
\special{fp}%
}}%
%
{{%
\special{pn 13}%
\special{pa 62 1438}%
\special{pa 2904 1438}%
\special{fp}%
\special{sh 1}%
\special{pa 2904 1438}%
\special{pa 2836 1418}%
\special{pa 2850 1438}%
\special{pa 2836 1458}%
\special{pa 2904 1438}%
\special{fp}%
}}%
{{%
\special{pn 13}%
\special{pn 13}%
\special{pa 290 2201}%
\special{pa 290 2187}%
\special{ip}%
\special{pa 290 2153}%
\special{pa 290 2140}%
\special{ip}%
\special{pa 290 2106}%
\special{pa 290 2092}%
\special{ip}%
\special{pa 290 2058}%
\special{pa 290 2044}%
\special{ip}%
\special{pa 290 2010}%
\special{pa 290 1997}%
\special{ip}%
\special{pa 290 1963}%
\special{pa 290 1949}%
\special{ip}%
\special{pa 290 1915}%
\special{pa 290 1901}%
\special{ip}%
\special{pa 290 1868}%
\special{pa 290 1854}%
\special{ip}%
\special{pa 290 1820}%
\special{pa 290 1806}%
\special{ip}%
\special{pa 290 1772}%
\special{pa 290 1759}%
\special{ip}%
\special{pa 290 1725}%
\special{pa 290 1711}%
\special{ip}%
\special{pa 290 1677}%
\special{pa 290 1663}%
\special{ip}%
\special{pa 290 1630}%
\special{pa 290 1616}%
\special{ip}%
\special{pa 290 1582}%
\special{pa 290 1568}%
\special{ip}%
\special{pa 290 1534}%
\special{pa 290 1520}%
\special{ip}%
\special{pa 290 1487}%
\special{pa 290 1473}%
\special{ip}%
\special{ip}%
\special{pa 290 1439}%
\special{pa 306 1439}%
\special{pa 310 1440}%
\special{pa 320 1441}%
\special{pa 326 1442}%
\special{pa 336 1444}%
\special{pa 340 1445}%
\special{pa 350 1448}%
\special{pa 390 1464}%
\special{pa 406 1472}%
\special{pa 420 1480}%
\special{pa 426 1483}%
\special{pa 466 1506}%
\special{pa 470 1509}%
\special{pa 490 1519}%
\special{pa 520 1531}%
\special{pa 526 1533}%
\special{pa 530 1534}%
\special{pa 550 1538}%
\special{pa 556 1539}%
\special{pa 586 1539}%
\special{pa 596 1538}%
\special{pa 600 1537}%
\special{pa 606 1536}%
\special{pa 610 1535}%
\special{pa 616 1534}%
\special{pa 620 1532}%
\special{pa 626 1530}%
\special{pa 646 1522}%
\special{pa 656 1517}%
\special{pa 660 1514}%
\special{pa 676 1505}%
\special{pa 680 1502}%
\special{pa 686 1498}%
\special{pa 690 1495}%
\special{pa 696 1491}%
\special{pa 716 1475}%
\special{pa 720 1471}%
\special{pa 726 1466}%
\special{pa 730 1462}%
\special{pa 736 1457}%
\special{pa 766 1427}%
\special{pa 770 1422}%
\special{pa 776 1416}%
\special{pa 780 1411}%
\special{pa 786 1405}%
\special{pa 836 1345}%
\special{pa 840 1339}%
\special{pa 846 1332}%
\special{pa 850 1326}%
\special{pa 856 1319}%
\special{pa 860 1313}%
\special{pa 866 1306}%
\special{pa 876 1293}%
\special{pa 880 1286}%
\special{pa 896 1266}%
\special{pa 900 1259}%
\special{pa 960 1175}%
\special{pa 966 1168}%
\special{pa 1016 1097}%
\special{pa 1020 1090}%
\special{pa 1060 1034}%
\special{pa 1066 1027}%
\special{pa 1096 986}%
\special{pa 1100 979}%
\special{pa 1110 966}%
\special{pa 1116 959}%
\special{pa 1126 946}%
\special{pa 1130 939}%
\special{pa 1136 933}%
\special{pa 1146 920}%
\special{pa 1150 914}%
\special{pa 1156 907}%
\special{pa 1160 901}%
\special{pa 1220 829}%
\special{pa 1226 823}%
\special{pa 1236 812}%
\special{pa 1240 806}%
\special{pa 1246 801}%
\special{pa 1256 790}%
\special{pa 1260 785}%
\special{pa 1306 740}%
\special{pa 1310 735}%
\special{pa 1320 726}%
\special{pa 1326 721}%
\special{pa 1330 717}%
\special{pa 1346 704}%
\special{pa 1350 700}%
\special{pa 1370 684}%
\special{pa 1376 680}%
\special{pa 1386 673}%
\special{pa 1390 669}%
\special{pa 1396 666}%
\special{pa 1400 662}%
\special{pa 1406 659}%
\special{pa 1446 635}%
\special{pa 1450 632}%
\special{pa 1470 622}%
\special{pa 1476 620}%
\special{pa 1506 608}%
\special{pa 1510 606}%
\special{pa 1516 605}%
\special{pa 1520 603}%
\special{pa 1526 602}%
\special{pa 1536 599}%
\special{pa 1540 598}%
\special{pa 1560 594}%
\special{pa 1566 593}%
\special{pa 1586 591}%
\special{pa 1626 591}%
\special{pa 1630 592}%
\special{pa 1636 592}%
\special{pa 1640 593}%
\special{pa 1650 594}%
\special{pa 1656 595}%
\special{pa 1666 597}%
\special{pa 1670 598}%
\special{pa 1680 601}%
\special{pa 1686 602}%
\special{pa 1690 604}%
\special{pa 1696 605}%
\special{pa 1700 607}%
\special{pa 1726 617}%
\special{pa 1756 632}%
\special{pa 1760 635}%
\special{pa 1766 638}%
\special{pa 1770 641}%
\special{pa 1786 651}%
\special{pa 1790 654}%
\special{pa 1796 658}%
\special{pa 1836 690}%
\special{pa 1840 694}%
\special{pa 1846 699}%
\special{pa 1886 739}%
\special{pa 1890 744}%
\special{pa 1896 750}%
\special{pa 1930 792}%
\special{pa 1936 798}%
\special{pa 1940 805}%
\special{pa 1950 818}%
\special{pa 1956 825}%
\special{pa 1970 846}%
\special{pa 1976 853}%
\special{pa 1996 883}%
\special{pa 2000 891}%
\special{pa 2020 923}%
\special{pa 2026 931}%
\special{pa 2030 940}%
\special{pa 2036 948}%
\special{pa 2040 957}%
\special{pa 2076 1020}%
\special{pa 2080 1029}%
\special{pa 2086 1039}%
\special{pa 2120 1109}%
\special{pa 2126 1119}%
\special{pa 2130 1130}%
\special{pa 2170 1218}%
\special{pa 2180 1241}%
\special{pa 2186 1253}%
\special{pa 2206 1301}%
\special{pa 2210 1313}%
\special{pa 2216 1326}%
\special{pa 2220 1338}%
\special{pa 2226 1351}%
\special{pa 2260 1442}%
\special{pa 2270 1469}%
\special{pa 2276 1483}%
\special{pa 2296 1539}%
\special{pa 2300 1553}%
\special{pa 2310 1582}%
\special{pa 2350 1702}%
\special{pa 2356 1717}%
\special{pa 2360 1733}%
\special{pa 2386 1813}%
\special{pa 2390 1829}%
\special{pa 2400 1862}%
\special{pa 2436 1981}%
\special{pa 2440 1998}%
\special{pa 2446 2016}%
\special{pa 2450 2033}%
\special{pa 2456 2051}%
\special{pa 2470 2105}%
\special{pa 2476 2123}%
\special{pa 2496 2197}%
\special{pa 2496 2201}%
\special{fp}%
}}%
\put(23.0000,-7.5000){\makebox(0,0){$g(V_c)$}}%
\end{picture}}%
  \caption{\small local max is positive}  \label{fig1} \end{minipage} \begin{minipage}{0.45\hsize}  
{\unitlength 0.1in
\begin{picture}( 28.7900, 19.9000)( -0.1400,-22.0000)
\put(1.8600,-15.2100){\makebox(0,0)[rt]{$O$}}%
\put(28.4300,-16.8600){\makebox(0,0){$V_c$}}%
%
{{%
\special{pn 13}%
\special{pa 284 2200}%
\special{pa 284 210}%
\special{fp}%
\special{sh 1}%
\special{pa 284 210}%
\special{pa 264 278}%
\special{pa 284 264}%
\special{pa 304 278}%
\special{pa 284 210}%
\special{fp}%
}}%
%
{{%
\special{pn 13}%
\special{pa 62 1438}%
\special{pa 2864 1438}%
\special{fp}%
\special{sh 1}%
\special{pa 2864 1438}%
\special{pa 2796 1418}%
\special{pa 2810 1438}%
\special{pa 2796 1458}%
\special{pa 2864 1438}%
\special{fp}%
}}%
{{%
\special{pn 13}%
\special{pa 290 1441}%
\special{pa 310 1441}%
\special{pa 316 1442}%
\special{pa 326 1442}%
\special{pa 330 1443}%
\special{pa 336 1444}%
\special{pa 340 1444}%
\special{pa 346 1445}%
\special{pa 386 1453}%
\special{pa 390 1454}%
\special{pa 400 1457}%
\special{pa 406 1458}%
\special{pa 410 1460}%
\special{pa 420 1463}%
\special{pa 426 1465}%
\special{pa 440 1470}%
\special{pa 446 1472}%
\special{pa 480 1486}%
\special{pa 486 1488}%
\special{pa 510 1499}%
\special{pa 516 1501}%
\special{pa 540 1512}%
\special{pa 546 1514}%
\special{pa 590 1532}%
\special{pa 596 1534}%
\special{pa 616 1541}%
\special{pa 620 1543}%
\special{pa 626 1544}%
\special{pa 630 1546}%
\special{pa 650 1552}%
\special{pa 656 1553}%
\special{pa 666 1556}%
\special{pa 670 1557}%
\special{pa 716 1566}%
\special{pa 720 1567}%
\special{pa 730 1568}%
\special{pa 736 1569}%
\special{pa 746 1570}%
\special{pa 750 1571}%
\special{pa 756 1571}%
\special{pa 766 1572}%
\special{pa 770 1572}%
\special{pa 780 1573}%
\special{pa 846 1573}%
\special{pa 860 1572}%
\special{pa 866 1572}%
\special{pa 876 1571}%
\special{pa 880 1571}%
\special{pa 890 1570}%
\special{pa 896 1570}%
\special{pa 900 1569}%
\special{pa 906 1569}%
\special{pa 910 1568}%
\special{pa 916 1568}%
\special{pa 920 1567}%
\special{pa 930 1566}%
\special{pa 936 1565}%
\special{pa 946 1564}%
\special{pa 950 1563}%
\special{pa 960 1562}%
\special{pa 966 1561}%
\special{pa 976 1560}%
\special{pa 980 1559}%
\special{pa 1000 1556}%
\special{pa 1006 1555}%
\special{pa 1026 1552}%
\special{pa 1030 1551}%
\special{pa 1050 1548}%
\special{pa 1116 1537}%
\special{pa 1136 1534}%
\special{pa 1140 1533}%
\special{pa 1156 1531}%
\special{pa 1160 1530}%
\special{pa 1170 1529}%
\special{pa 1176 1528}%
\special{pa 1190 1526}%
\special{pa 1196 1525}%
\special{pa 1200 1525}%
\special{pa 1206 1524}%
\special{pa 1216 1523}%
\special{pa 1220 1522}%
\special{pa 1226 1522}%
\special{pa 1230 1521}%
\special{pa 1236 1521}%
\special{pa 1240 1520}%
\special{pa 1246 1520}%
\special{pa 1250 1519}%
\special{pa 1256 1519}%
\special{pa 1260 1518}%
\special{pa 1266 1518}%
\special{pa 1270 1517}%
\special{pa 1276 1517}%
\special{pa 1286 1516}%
\special{pa 1290 1516}%
\special{pa 1306 1515}%
\special{pa 1310 1515}%
\special{pa 1336 1514}%
\special{pa 1380 1514}%
\special{pa 1400 1515}%
\special{pa 1406 1515}%
\special{pa 1420 1516}%
\special{pa 1426 1516}%
\special{pa 1430 1517}%
\special{pa 1436 1517}%
\special{pa 1440 1518}%
\special{pa 1446 1518}%
\special{pa 1450 1519}%
\special{pa 1456 1519}%
\special{pa 1460 1520}%
\special{pa 1466 1520}%
\special{pa 1470 1521}%
\special{pa 1486 1523}%
\special{pa 1490 1524}%
\special{pa 1550 1536}%
\special{pa 1556 1537}%
\special{pa 1570 1541}%
\special{pa 1576 1542}%
\special{pa 1580 1544}%
\special{pa 1586 1545}%
\special{pa 1590 1547}%
\special{pa 1596 1548}%
\special{pa 1600 1550}%
\special{pa 1610 1553}%
\special{pa 1616 1555}%
\special{pa 1626 1558}%
\special{pa 1630 1560}%
\special{pa 1676 1578}%
\special{pa 1680 1580}%
\special{pa 1696 1587}%
\special{pa 1736 1607}%
\special{pa 1750 1615}%
\special{pa 1756 1618}%
\special{pa 1796 1642}%
\special{pa 1800 1645}%
\special{pa 1810 1652}%
\special{pa 1816 1655}%
\special{pa 1826 1662}%
\special{pa 1830 1665}%
\special{pa 1836 1669}%
\special{pa 1846 1676}%
\special{pa 1850 1680}%
\special{pa 1860 1687}%
\special{pa 1866 1691}%
\special{pa 1906 1723}%
\special{pa 1910 1727}%
\special{pa 1920 1736}%
\special{pa 1926 1740}%
\special{pa 1936 1749}%
\special{pa 1940 1753}%
\special{pa 1946 1758}%
\special{pa 1956 1767}%
\special{pa 1960 1772}%
\special{pa 2020 1832}%
\special{pa 2026 1837}%
\special{pa 2036 1848}%
\special{pa 2040 1853}%
\special{pa 2046 1859}%
\special{pa 2056 1870}%
\special{pa 2060 1876}%
\special{pa 2076 1893}%
\special{pa 2080 1899}%
\special{pa 2106 1929}%
\special{pa 2110 1935}%
\special{pa 2126 1954}%
\special{pa 2130 1960}%
\special{pa 2136 1967}%
\special{pa 2140 1973}%
\special{pa 2146 1980}%
\special{pa 2150 1986}%
\special{pa 2156 1993}%
\special{pa 2170 2013}%
\special{pa 2176 2020}%
\special{pa 2206 2062}%
\special{pa 2210 2069}%
\special{pa 2226 2091}%
\special{pa 2230 2098}%
\special{pa 2236 2106}%
\special{pa 2240 2113}%
\special{pa 2246 2121}%
\special{pa 2256 2136}%
\special{pa 2260 2144}%
\special{pa 2296 2199}%
\special{pa 2296 2201}%
\special{fp}%
}}%
\put(23.8000,-18.6000){\makebox(0,0){$g(V_c)$}}%
\end{picture}}%
       \caption{\small local max is negative}  \label{fig2} \end{minipage}\end{figure}

For the analysis in the rest of the paper it is convenient to 
write the system \eqref{r0}  as 

\begin{equation}\lb{sp0}
\sF_j(\lambda,\rho_i,R_e,V)=0 \qu \text{for $j=1,2,3,4$,}  
\end{equation}
where we denote $\lambda=V_c/L$ and 
\begin{align*}
\sF_1
:=&k_i\partial_x\left\{(\partial_x V+\lambda)\rho_i\right\}
-k_eh\left(\partial_x V+\lambda\right)e^{-\frac{\lambda}{2}x}R_e,
\\
\sF_2
:=&-\partial_x^2R_e 
-(\partial_x V) \partial_x R_e
+\left\{\frac{\lambda}{2}\partial_x V
-\partial_x^2V
+\frac{\lambda^2}{4}
-h\left(\partial_x V+\lambda\right)\right\}R_e,
\\
\sF_3
:=&\partial_x^2 V-\rho_i+e^{-\frac{\lambda}{2}x}R_e,
\\
\sF_4
:=&\partial_xR_e(L)+\left(\partial_x V(L)+\frac{\la}{2}\right)R_e(L)-
\gamma\frac{k_i}{k_e}e^{\frac{\la}{2}L}
\left(\partial_x V(L)+{\la}\right)\ro_i(L).
\end{align*}


\section {Bifurcation}\lb{S2}
In this section we apply the following well-known theorem  \cite{CR1} on bifurcation from a simple eigenvalue.  
Let $N(\sL)$ and $R(\sL)$ denote the nullspace and range of any linear operator $\sL$ between
two Banach spaces.
\begin{thm}     \lb{Local}
Let $X$ and $Y$ be Banach spaces,
$\sO$ be an open subset of ${\mathbb R}\times X$ and $\sF:\sO \to Y$ be a $C^2$ function. 
Suppose that
\begin{enumerate}[{(H}1{)}]
\item $(\lambda,0)\in \sO$ and $\sF(\lambda,0)=0$ for all $\lambda \in \mathbb R$;
\item for some $\lambda^* \in \mathbb R$, $N(\partial_u\sF(\lambda^*,0))$ and 
$Y\backslash R(\partial_u\sF(\lambda^*,0))$ are one-dimensional, 
with the null space generated by $u^*$, which satisfies  the transversality condition
\begin{equation*} 
\partial_\la\partial_u \sF(\lambda^*,0)(1,u^*)\notin R(\partial_u\sF(\lambda^*,0)),  
\end{equation*}  
where $\partial_u$ and $\partial_\la\partial_u$ denote Fr\'echet derivatives for $(\la,u) \in \sO$.
\end{enumerate}
Then there exists in $\sO$ a continuous curve 
${\sK}=\{(\la(s),u(s));s \in \mathbb R\}$
of solutions of the equation $\sF(\lambda,u)=0$ such that:
\begin{enumerate}[{(C}1{)}]
\item $(\la(0),u(0))=(\la^*,0)$;
\item $u(s)=su^*+o(s)$ in $X$  as $s \to 0$;
\item there exists a neighborhood $\sW$ of $(\la^*,0)$
and $\ve > 0$ sufficiently small such that
\begin{equation*}
\{(\la,u)\in \sW ; u\neq 0 \text{ and } \sF(\la,u)=0 \}
=\{(\la(s),u(s)) ; 0<|s|<\ve\}.
\end{equation*}
\end{enumerate}
\end{thm}

In order to apply the theorem to our situation, we use the notation $u=(\rho_i,R_e,V)$ and 
we define the two spaces 
\begin{align*}
X: \ &\rho_i \in \{f\in C^1([0,L]);f(0)=0\}, \ \ R_e \in \{f\in C^2([0,L]);f(0)=0\},
\\
& V \in \{f\in C^3([0,L]);f(0)=f(L)=0\};
\\
Y: \ &\sF_1 \in C^0([0,L]), \ \ \sF_2 \in C^0([0,L]),
\ \ \sF_3 \in C^1([0,L]), \ \ \sF_4 \in {\mathbb R}
\end{align*}  
and the sets
\begin{align*}  
\sO:=&\{(\lambda,\rho_i,R_e,V)\in (0,\infty)\times X;\ \partial_x V+\lambda>0\} \ 
=\ \bigcup_{j\in \bbn} \sO_j ,\text{ where } 
\\
\sO_j:=&\{(\lambda,\rho_i,R_e,V)\in (0,\infty)\times X; \ 
\lambda+\|(\rho_i,R_e,V)\|_{X}\leq j ,  \  
\lambda\geq \tfrac1j,  \ \partial_x V+\lambda\geq \tfrac1j\}.
\end{align*}  
Note that $\sO$ is an open set and each $\sO_j$ is a closed bounded subset of $\sO$.
Furthermore, the $\sF_j$ are real-analytic operators because they are 
polynomials in $(\lambda,\rho_i, R_e, V)$ and their $x$-derivatives, 
except for the factor $h(\pa_xV+\lambda)$.  However, $\pa_xV+\lambda > 0$ 
in $\sO$ and the function $s\to h(s)$ is analytic for $s>0$.
Hypothesis $(H1)$ is obvious.  
The local bifurcation condition $(H2)$ is verified in Lemmas \ref{null1}--\ref{trans1}.

\begin{lem}\lb{null1}
Recall that $\la=V_c/L$.  Let $\sL= \partial_{(\rho_i,R_e,V)}\sF(\lambda,0,0,0)$ be the linearized operator around 
a trivial solution and let $N(\sL)$ be its nullspace. Then 

\begin{enumerate}[{(a)}]

\item $N(\sL)$ is at most one-dimensional 
for any $\lambda>0$.

\item $N(\sL)$ is one-dimensional if and only if $D(V_c)=0$.  
 Thus the sparking voltage exists.

\item  $N(\sL)$ has a  basis $(\varphi_i,\varphi_e,\varphi_v)$ with 
\begin{equation}\lb{positive0}
\varphi_i(x)>0, \quad \varphi_e(x)>0 \quad \text{for $x\in (0,L]$} 
\end{equation}
if and only if 
\begin{equation}\lb{positive1}
D(V_c)=0, \quad g(V_c) < \frac{\pi^2}{L^2}. 
\end{equation}

\item $V_c^\dagger$ defined in \er{bp1} satisfies \er{positive1}.  
\end{enumerate}
\end{lem}
\begin{proof}
We remark that the positivity \eqref{positive0} will lead to the positivity of $R_e$ and $\rho_i$ 
in the local bifurcation proof.

(a) If $(S_i,S_e,W) \in N(\sL) \subset X$,
then $(S_i,S_e,W)$ solves 
\begin{align}
\partial_{(\rho_i,R_e,V)}\sF_1(\lambda,0,0,0)[S_i,S_e,W]&=
k_i\lambda\partial_xS_i
-k_eh\left(\lambda\right)e^{-\frac{\lambda}{2}x}S_e=0,
\lb{LinearEq1}\\
\partial_{(\rho_i,R_e,V)}\sF_2(\lambda,0,0,0)[S_i,S_e,W]&=
-\partial_x^2S_e-g(\lambda L)S_e=0,
\lb{LinearEq2}\\
\partial_{(\rho_i,R_e,V)}\sF_3(\lambda,0,0,0)[S_i,S_e,W]&=
\partial_x^2 W-S_i+e^{-\frac{\lambda}{2}x}S_e=0,
\lb{LinearEq3}\\
\partial_{(\rho_i,R_e,V)}\sF_4(\lambda,0,0,0)[S_i,S_e,W]&=
\partial_xS_e(L)+\frac{\la}{2}S_e(L)-
\gamma\frac{k_i}{k_e}\lambda e^{\frac{\la}{2}L}S_i(L)=0.
\lb{LinearEq4}
\end{align}
Solving \er{LinearEq1} with $S_i(0)=0$, we have
\begin{equation}\lb{Sol1}
\frac{k_i}{k_e}\lambda S_i(x)
=h\left(\lambda\right)\int_0^x e^{-\frac{\lambda}{2}y}S_e(y)\,dy.
\end{equation}
By \er{Sol1} with $x=L$,
we rewrite the boundary condition \er{LinearEq4} so that
\begin{equation}\lb{LinearBC1}
\partial_xS_e(L)+\frac{\la}{2}S_e(L)=
\gamma h\left(\lambda\right)e^{\frac{\lambda}{2}L}
\int_0^L e^{-\frac{\lambda}{2}y}S_e(y)\,dy,
\end{equation}
which is closed with respect to $S_e$.
Therefore, we have a differential equation for $S_e$ 
with two boundary conditions.  
It suffices to solve it in order to obtain all elements of the nullspace. 
Indeed, $S_i$ is obtained by \er{Sol1} and $S_e$ and 
$W$ is obtained by solving  \er{LinearEq3} with $W(0)=W(L)=0$.
The general solutions of the second order equation \er{LinearEq2} with $S_e(0)=0$ are 
\begin{equation}\lb{Sol2}
S_e(x)=\left\{
\begin{array}{ll}
A \sinh\sqrt{-g(\lambda L)}x & \text{if $g(\lambda L)<0$},
\\
A x & \text{if $g(\lambda L)=0$},
\\
A \sin\sqrt{g(\lambda L)} x & \text{if $g(\lambda L)>0$},
\end{array}
\right.
\end{equation}
where we have also used the the boundary condition $S_e(0)=0$.
This fact means that the null space $N(\partial_{(\rho_i,R_e,V)}\sF(\lambda,0,0,0))$ 
is at most one-dimensional for any $\lambda>0$.

(b) We will show that
equation \er{LinearEq2} with $S_e(0)=0$ and \er{LinearBC1}
admits nontrivial solutions if and only if $D(V_c)=0$.  
 We write $g=g(\lambda L)$ and first treat the case $g<0$.
To this end, we substitute the general solution into \er{LinearBC1}
and see that a necessary and sufficient condition 
for the existence of nontrivial solutions is 
\begin{align*}
0=&\sqrt{-g}\cosh\sqrt{-g}L  
+\frac{\lambda}{2}\sinh\sqrt{-g}L 
-\gamma h\left(\lambda\right)e^{\frac{\lambda}{2}L}
\int_0^L e^{-\frac{\lambda}{2}y}\sinh\sqrt{-g}y \,dy
\\
=&(1+\gamma)\left\{\sqrt{-g}\cosh\sqrt{-g}L
+\frac{\lambda}{2}\sinh\sqrt{-g}L\right\}
+2\gamma \sqrt{-g}e^{\frac{\lambda}{2}L}.
\end{align*}
In deriving the last equality, we have also used the fact $g+\frac{\la^2}{4}=h(\lambda)$.
This equality is equivalent to $D(V_c)=0$.
Now we consider the case $g=0$. 
As above, we find the condition
\begin{align*}
0=1+\frac{\lambda}{2}L
-\gamma h\left(\lambda\right)e^{\frac{\lambda}{2}L}
\int_0^L e^{-\frac{\lambda}{2}y}y\,dy
=(1+\gamma)\left(1+\frac{\lambda}{2}L\right)
-\gamma e^{\frac{\lambda}{2}}.
\end{align*}
This too is equivalent to $D(V_c)=0$.
For the case $g>0$, we have 
\begin{align*}
0=&\sqrt{g}\cos\sqrt{g}L
+\frac{\lambda}{2}\sin\sqrt{g}L
-\gamma h\left(\lambda\right)e^{\frac{\lambda}{2}L}
\int_0^L e^{-\frac{\lambda}{2}y}\sin\sqrt{g}y\,dy
\\
=&(1+\gamma)\left(\sqrt{g}\cos\sqrt{g}L
+\frac{\lambda}{2}\sin\sqrt{g}L\right)
-\gamma \sqrt{g}e^{\frac{\lambda}{2}L}. 
\end{align*}
Once again this is equivalent to $D(V_c)=0$. 
Thus we conclude in all three cases that $N(\sL)$ is one-dimensional 
if and only if $D(V_c)=0$.  

(c) Furthermore, it is seen from \er{Sol2} that
the null space $N(\sL)$ has 
a basis with \er{positive0} if and only if \er{positive1} holds.

(d)  It remains to show that 
the sparking voltage $V_c^\dagger$ must satisfy \er{positive1}.
Suppose on the contrary that $g(V_c^\dagger) \geq \pi^2/L^2$ holds.
Then the graph of $g$ must be drawn as in Figure \ref{fig1}.
Therefore, there exists a positive constant $V_c^* \leq V_c^\dagger $ such that
$g(V_c) <  \pi^2/L^2$ for all $V_c \in [0,V_c^*)$
and $g(V_c^*)= \pi^2/L^2$. 
Evaluating $D(V_c)$ at $V_c=V_c^*$, we see
that 
\[
D(V_c^*)=\cos\sqrt{g(V_c^*)}L
+\frac{V_c^*}{2\sqrt{g(V_c^*)}L}\sin\sqrt{g(V_c^*)}L
-\frac{\gamma}{1+\gamma}e^{\frac{V_c^*}{2}}
=-1-\frac{\gamma}{1+\gamma}e^{\frac{V_c^*}{2}}<0.
\]
However,  $\lim_{V_c \to 0}D(V_c)=1/(1+\gamma)>0$.
These facts together with the intermediate value theorem means that
there exists $0<c_0<V_c^*$ such that $D(c_0)=0$, 
so that $V_c^\dagger$ is not the smallest root of $D$, which contradicts its definition.  
\end{proof}

In order to apply Theorem \ref{Local}, we define $\la^* = V_c^\dagger/L$ and we let 
$u^*=(\varphi_i^\dagger,\varphi_e^\dagger,\varphi_v^\dagger)$ denote
a basis of $N(\partial_{(\rho_i,R_e,V)}\sF(V_c^\dagger/L,0,0,0))$ that satisfies \er{positive0}.

\begin{lem}\lb{range1}
The quotient space $Y\backslash R(\sL)$ 
is at most one-dimensional. 
Furthermore, it is one-dimensional if and only if $D(V_c)=0$.
\end{lem}
\begin{proof}
Let us denote $\partial_{(\rho_i,R_e,V)}\sF(V_c/L,0,0,0))$ by $\sL$.
We begin by representing the range as 
\begin{gather}
R(\sL)=\left\{(f_i,f_e,f_v,f_b) \in Y; \er{Orthogonal1} \right\}, 
\lb{reperesent1}\\
\int_0^L (f_i\psi_i + f_e\psi_e + f_v\psi_v) dx + f_b\psi_b=0 \ \
\text{for all $(\psi_i,\psi_e,\psi_v,\psi_b)\in N((\sL\bl)^*)$}. 
\lb{Orthogonal1}
\end{gather}
Here $(\sL\bl)^*$ is defined conveniently on a Hilbert space as follows.
Let $X^\bullet$ be the same as $X$ except that $C^k$ is replaced by $H^k$ for $k=1,2,3$.  
Let $Y\bl$ be the same as $Y$ except that $C^k$ is replaced by $H^k$ for $k=0,1$.  
Define $\sL\bl :X\bl \to Y\bl$ to be the unique linear extension of $\sL$ to $X\bl$,
and $(\sL\bl)^*$ to be the adjoint operator of $\sL\bl$.
By standard operator theory, 
\begin{gather*}
R(\sL\bl)=\{(f_i,f_e,f_v,f_b)\in Y\bl; \er{Orthogonal1}\}.
\end{gather*}
From this and the fact $Y\subset Y\bl$, \er{Orthogonal1} is necessary
for the solvability of the problem $\sL\bl(S_i,S_e,W)=(f_i,f_e,f_v,f_b) \in Y$.
On the other hand, if $(f_i,f_e,f_v,f_b) \in \{f\in Y ; \er{Orthogonal1}\}$,
we have a unique solution $(S_i,S_e,W) \in X\bl$ to 
the problem $\sL\bl(S_i,S_e,W)=(f_i,f_e,f_v,f_b) \in Y$.
Then  $(S_i,S_e,W) \in X$ by standard elliptic estimates.
These facts lead to the representation \er{reperesent1}.

It remains to prove that
$N((\sL\bl)^*)$ is at most one-dimensional,
and it is one-dimensional if and only if $D(V_c)=0$.
We first claim that the operator $(\sL\bl)^*$ is precisely given by 
\begin{subequations}
\begin{align}
D((\sL\bl)^*):=&\{(\psi_i,\psi_e,\psi_v,\psi_b)\in H^1(I)\times H^2(I) \times H^3(I) \times \mathbb R; \er{domain2} \text{ holds}  \},
\lb{domain1}\\
(\sL\bl)^*_1(\psi_i,\psi_e,\psi_v,\psi_b):=&
-k_i\lambda \partial_x\psi_i-\psi_v,
\lb{AdjointOp1}\\
(\sL\bl)^*_2(\psi_i,\psi_e,\psi_v,\psi_b):=&
-\partial_x^2\psi_e-g(\lambda L)\psi_e
-k_eh(\lambda)e^{-\frac{\lambda}{2}x} \psi_i+e^{-\frac{\lambda}{2}x}\psi_v,
\lb{AdjointOp2}\\
(\sL\bl)^*_3(\psi_i,\psi_e,\psi_v,\psi_b):=&
\partial_x^2\psi_v,
\lb{AdjointOp3}
\end{align}
where 
\begin{equation}\lb{domain2}
\psi_e(L)-\psi_b=\psi_e(0)=\psi_v(0)=\psi_v(L)
=k_i\lambda \psi_i(L) -\gamma\frac{k_i}{k_e}\lambda e^{\frac{\la}{2}L}\psi_b
=\partial_x\psi_e(L)+\frac{\lambda}{2}\psi_b=0.
\end{equation}
\end{subequations}
We now verify the claim.      It suffices to check that
\begin{align*}
&\langle (\sL\bl_1(S_i,S_e,W),\sL\bl_2(S_i,S_e,W),\sL\bl_3(S_i,S_e,W)),(\psi_i,\psi_e,\psi_v) \rangle+\sL\bl_4(S_i,S_e,W)\psi_b
\\
&=\langle (S_i,S_e,W),(\sL\bl)^*(\psi_i,\psi_e,\psi_v,\psi_b) \rangle
\end{align*}
for all $(S_i,S_e,W) \in X\bl$ and 
$(\psi_i,\psi_e,\psi_v,\psi_b) \in D((\sL\bl)^*)$,
where $\langle\cdot,\cdot\rangle$ denotes the inner product of $L^2(I)$.
We observe that
\begin{align*}
&\langle (\sL\bl_1(S_i,S_e,W),\sL\bl_2(S_i,S_e,W),\sL\bl_3(S_i,S_e,W)),(\psi_i,\psi_e,\psi_v) \rangle+\sL\bl_4(S_i,S_e,W)\psi_b
\\
&=-\langle S_i, k_i\lambda \partial_x\psi_i \rangle 
+S_i(L) k_i\lambda \psi_i(L)
-\langle S_e, k_eh(\lambda)e^{-\frac{\lambda}{2}x} \psi_i \rangle
\\
&\quad -\langle S_e, \partial_x^2\psi_e \rangle
-\partial_xS_e(L)\psi_e(L)
+\partial_xS_e(0)\psi_e(0)
+S_e(L)\partial_x\psi_e(L)
-\langle S_e, g(\lambda L)\psi_e \rangle 
\\
&\quad +\langle W, \partial_x^2\psi_v \rangle
+\partial_x W(L)\psi_v(L)-\partial_x W(0)\psi_v(0)
-\langle S_i, \psi_v \rangle 
+\langle S_e, e^{-\frac{\lambda}{2}x}\psi_v \rangle
\\ 
&\quad +\partial_xS_e(L)\psi_b+S_e(L)\frac{\lambda}{2}\psi_b
-S_i(L)\gamma\frac{k_i}{k_e}\lambda e^{\frac{\la}{2}L}\psi_b, 
\end{align*}
due to integration by parts and  $S_i(0)=S_e(0)=W(0)=W(L)=0$.  
Grouping them with respect to $S_i$, $S_i(L)$, $S_e$, $\partial_xS_e(L)$,
$\partial_xS_e(0)$, $S_e(L)$, $W$, $\partial_xW(L)$, and $\partial_xW(0)$, 
and also using the boundary conditions \er{domain2}, we have
\begin{align*}
&\langle (\sL\bl_1(S_i,S_e,W),\sL\bl_2(S_i,S_e,W),\sL\bl_3(S_i,S_e,W)),(\psi_i,\psi_e,\psi_v) \rangle+\sL\bl_4(S_i,S_e,W)\psi_b
\\
&=-\langle S_i, k_i\lambda \partial_x\psi_i+\psi_v \rangle
-\langle S_e, k_eh(\lambda)e^{-\frac{\lambda}{2}x} \psi_i 
+\partial_x^2\psi_e+g(\lambda L)\psi_e- e^{-\frac{\lambda}{2}x}\psi_v\rangle
+\langle W, \partial_x^2\psi_v \rangle
\\
&=\langle (S_i,S_e,W),(\sL\bl)^*(\psi_i,\psi_e,\psi_v,\psi_b) \rangle.
\end{align*}
This proves the claim. 

Next we compute $N((\sL\bl)^*)$.   
To this end, 
we seek solutions
$(\psi_i,\psi_e,\psi_v,\psi_b) \in D((\sL\bl)^*)$ to the problem 
$(\sL\bl)^*(\psi_i,\psi_e,\psi_v,\psi_b)=0$. 
From $(\sL\bl)^*_3=0$ and boundary conditions $\psi_v(0)=\psi_v(L)=0$,
we see that 
\begin{equation}\lb{psi_v}
\psi_v=0.
\end{equation}
From this and $(\sL\bl)^*_1=(\sL\bl)^*_2=0$, we have the equations
\begin{subequations}\lb{AdjointProblem}
\begin{gather}
\partial_x\psi_i=0,
\lb{AdjointEq1}\\
-\partial_x^2\psi_e-g(\lambda L)\psi_e-k_eh(\lambda)e^{-\frac{\lambda}{2}x} \psi_i=0.
\lb{AdjointEq2}
\end{gather}
Owing to \er{domain2} and substituting $\psi_b=\psi_e(L)$, 
the boundary conditions for this system are  
\begin{gather}
\psi_i(L)-\frac{\gamma}{k_e} e^{\frac{\la}{2}L}\psi_e(L)=0,
\lb{AdjointBC1}
\\
\psi_e(0)=0,
\lb{AdjointBC2}
\\
\partial_x\psi_e(L)+\frac{\lambda}{2}\psi_e(L)=0.
\lb{AdjointBC3}
\end{gather}
\end{subequations}
Now it remains to solve the problem \er{AdjointProblem}
in order to check the null of $N((\sL\bl)^*)$.

Let us reduce the problem \er{AdjointProblem}
to a problem to a scalar equation for $\psi_e$ alone.
Integrating \er{AdjointEq1} over $[x,L]$ and using \er{AdjointBC1}, we obtain
\begin{equation}\lb{AdjointEq3}
\psi_i(x)=\frac{\gamma}{k_e} e^{\frac{\la}{2}L}\psi_e(L).
\end{equation}
Plugging 
this into \er{AdjointEq2}, we have the problem for $\psi_e$:
\begin{equation}
-\partial_x^2\psi_e-g(\lambda L)\psi_e=
\gamma h(\lambda) e^{\frac{\la}{2}L}\psi_e(L)e^{-\frac{\lambda}{2}x},
\lb{AdjointEq4} 
\end{equation}
together with \er{AdjointBC2} and \er{AdjointBC3}.

Then, regarding $\psi_e(L)$ on the left hand side of \er{AdjointEq4} as a given value, 
we have general solutions to \er{AdjointEq4}:
\begin{equation}\lb{AdjointSol1}
\psi_e=\left\{
\begin{array}{ll}
A e^{\sqrt{-g(\lambda L)}x}+Be^{-\sqrt{-g(\lambda L)}x} 
-\gamma e^{\frac{\la}{2}L}\psi_e(L)e^{-\frac{\lambda}{2}x}
& \text{if $g(\lambda L)<0$},
\\
A x+B -\gamma e^{\frac{\la}{2}L}\psi_e(L)e^{-\frac{\lambda}{2}x}
& \text{if $g(\lambda L)=0$},
\\
A \sin\sqrt{g(\lambda L)}x + B\cos\sqrt{g(\lambda L)}x
-\gamma e^{\frac{\la}{2}L}\psi_e(L)e^{-\frac{\lambda}{2}x}
& \text{if $g(\lambda L)>0$}.
\end{array}
\right.
\end{equation}
We do a separate but similar calculation in each case.  

\medskip

\noindent
\underline{\it Case $g<0$.} \ We write $g=g(\la L)$ and put $x=L$ in \er{AdjointSol1}.
Then we see that 
\[
\psi_e(L) =\frac{1}{1+\gamma} (A e^{\sqrt{-g}L}+Be^{-\sqrt{-g}L}).
\]
This and \er{AdjointBC2} give 
\begin{align*}
&0=\psi_e(0) = A +B - \frac{\gamma}{1+\gamma} e^{\frac{\la}{2}L} (A e^{\sqrt{-g}L}+Be^{-\sqrt{-g}L}). 
\end{align*}
Furthermore, from \er{AdjointBC3} and 
$\partial_x(e^{-\frac{\lambda}{2}x})+\frac{\lambda}{2}e^{-\frac{\lambda}{2}x}=0$, 
it must hold that
\begin{align*}
0=\partial_x{\psi}_e(L)+\frac{\lambda}{2}{\psi}_e(L)
=\sqrt{-g}\left(A e^{\sqrt{-g}L}-Be^{-\sqrt{-g}L}\right)
+\frac{\lambda}{2}\left(A e^{\sqrt{-g}L}+Be^{-\sqrt{-g}L} \right).
\end{align*}
Summarizing these two, we have a linear system for the pair $(A,B)$:
\[
M^-\begin{bmatrix} A \\ B \end{bmatrix} = \begin{bmatrix} 0 \\ 0 \end{bmatrix}, \quad
M^-:=\begin{bmatrix}
1-\frac{\gamma}{1+\gamma}e^{\frac{\la}{2}L} e^{\sqrt{-g}L} &
1-\frac{\gamma}{1+\gamma}e^{\frac{\la}{2}L} e^{-\sqrt{-g}L}
\\
\sqrt{-g}e^{\sqrt{-g}L}+\frac{\lambda}{2} e^{\sqrt{-g}L} &
-\sqrt{-g}e^{-\sqrt{-g}L}+\frac{\lambda}{2} e^{-\sqrt{-g}L}
\end{bmatrix}.
\]
It has nontrivial solutions if and only if $\det M^-=0$. 
Then the kernel is one-dimensional since $m^-_{21}$ is positive.
On the other hand, it holds that
\begin{align*}
\det M^-&=\left(-\sqrt{-g}e^{-\sqrt{-g}L}+\frac{\lambda}{2} e^{-\sqrt{-g}L} 
-\sqrt{-g}e^{\sqrt{-g}L}-\frac{\lambda}{2} e^{\sqrt{-g}L}\right)
\\
&\quad +\frac{\gamma}{1+\gamma}e^{\frac{\la}{2}L}
\left(\sqrt{-g}-\frac{\lambda}{2}+\sqrt{-g}+\frac{\lambda}{2}\right)
\\
&=-2\sqrt{-g}D(V_c).
\end{align*}
Hence we conclude that 
$N((\sL\bl)^*)$ is at most one-dimensional,
and it is one-dimensional if and only if $D(V_c)=0$.

\medskip

\noindent
\underline{\it Case $g=0$.} \  Putting $x=L$ in \er{AdjointSol1}, we have 
$\psi_e(L) = A L +B - \gamma \psi_e(L)$.
In the same way as above, using \er{AdjointBC2} and \er{AdjointBC3},
we have 
\[
M^0\begin{bmatrix} A \\ B \end{bmatrix} = \begin{bmatrix} 0 \\ 0 \end{bmatrix}, \quad
M^0:=\begin{bmatrix}
-\frac{\gamma}{1+\gamma}e^{\frac{\la}{2}L}L &
1-\frac{\gamma}{1+\gamma}e^{\frac{\la}{2}L}
\\
1+\frac{\lambda}{2}L & 
\frac{\lambda}{2}
\end{bmatrix}.
\]
Note that $1+\frac{\lambda}{2}L >0$ and 
\begin{align*}
\det M^0&=-\left(1+\frac{\lambda}{2}L\right)
+\frac{\gamma}{1+\gamma}e^{\frac{\la}{2}L}
\left(-\frac{\lambda}{2}L+1+\frac{\lambda}{2}L\right)
=-D(V_c).
\end{align*}
Hence we conclude that
$N((\sL\bl)^*)$ is at most one-dimensional,
and it is one-dimensional if and only if $D(V_c)=0$.

\medskip

\noindent
\underline{\it Case $g>0$.} \  We write $g=g(\la L)$ and put $x=L$ in \er{AdjointSol1}.
Then  
\[
\psi_e(L) =\frac{1}{1+\gamma} (A\sin\sqrt{g}L+B\cos\sqrt{g}L).
\]
This, \er{AdjointBC2} and \er{AdjointBC3} give us the identities
\begin{gather*}
0=\psi_e(0) =B - \frac{\gamma}{1+\gamma} e^{\frac{\la}{2}L} (A\sin\sqrt{g}L+B\cos\sqrt{g}L),
\\
0=\partial_x{\psi}_e(L)+\frac{\lambda}{2}{\psi}_e(L)
=\sqrt{g}(A\cos\sqrt{g}L-B\sin\sqrt{g}L)
+\frac{\lambda}{2}(A\sin\sqrt{g}L+B\cos\sqrt{g}L).
\end{gather*}
Summarizing these two, we have a linear equation for $(A,B)$:
\[
M^+\begin{bmatrix} A \\ B \end{bmatrix} = \begin{bmatrix} 0 \\ 0 \end{bmatrix}, \quad
M^+:=\begin{bmatrix}
-\frac{\gamma}{1+\gamma}e^{\frac{\lambda}{2}L}\sin\sqrt{g}L &
1-\frac{\gamma}{1+\gamma}e^{\frac{\lambda}{2}L}\cos\sqrt{g}L
\\
\sqrt{g}\cos\sqrt{g}L+\frac{\lambda}{2}\sin\sqrt{g}L &
-\sqrt{g}\sin\sqrt{g}L+\frac{\lambda}{2}\cos\sqrt{g}L
\end{bmatrix}.
\]
But note that
\begin{align*}
&\det M^+ 
=-\sqrt{g}\cos\sqrt{g}L+\frac{\lambda}{2}\sin\sqrt{g}L
\\
\quad &+\frac{\gamma}{1+\gamma}e^{\frac{\la}{2}L}
\left\{\sin\sqrt{g}L\left(\sqrt{g}\sin\sqrt{g}L-\frac{\lambda}{2}\cos\sqrt{g}L\right)
+\cos\sqrt{g}L\left(\sqrt{g}\cos\sqrt{g}L+\frac{\lambda}{2}\sin\sqrt{g}L\right)\right\}
\\
&=-\sqrt{g}D(V_c).
\end{align*}
Hence we conclude that 
$N((\sL\bl)^*)$ is at most one-dimensional,
and it is one-dimensional if and only if $D(V_c)=0$.
\end{proof}

In order to clarify the variables in the next lemma, we denote $D(V_c)=D(V_c,a,b,\gamma)$. 
Let us also define 
$$A = \{(a,b,\gamma)\in(\real_+)^3 \ ; \  \text{there exists a root of } D(V_c,a,b,\gamma)=0\}.$$  
Then by definition $V_c^\dagger = V_c^\dagger(a,b,\gamma)$ is the smallest root, for any $(a,b,\gamma)\in A$.  
Let $A^\circ$ be the interior of $A$.  
We also explicitly denote $g(V_c) = g(V_c,a,b) = \frac{aV_c}{L} \exp\frac{-bL}{V_c} - \frac{V_c^2}{4L^2}$.  
Transversality is the condition that the tangent of the presumed local curve and the  
tangent of the trivial curve do not coincide.

\begin{lem}\lb{trans1}
The transversality condition 
\begin{equation}\label{transversality1}
\partial_\la \pa_{(\rho_i,R_e,V)}  \sF(V_c^\dagger/L,0,0,0)[1,\varphi_i^\dagger,\varphi_e^\dagger,\varphi_v^\dagger]
\notin R(\partial_{(\rho_i,R_e,V)}\sF(V_c^\dagger/L,0,0,0))  
\end{equation}
is valid for almost every $(a,b,\gamma)\in A^{\circ}$. 
\end{lem}

\begin{proof} 				
The first part of the proof is devoted to showing that various sets of points $(a,b,\gamma)$ have measure zero in $\real^3$. 
It is easy to check that $g(V_c,a,b)=0$ has no solution if $a<\frac{e}{4}b$, 
exactly one solution $W_0(a,b)$ if $a=\frac{e}4b$, and exactly two solutions 
$W_1(a,b) , W_2(a,b)$ if $a>\frac{e}{4}b$.  
The set $Z_0 =  \{(a,b,\gamma)\in (\mathbb R_+)^{3}\ ;\ a=\frac{e}{4}b\}$ obviously has measure zero. 
On its complement $Z_0^c$ we calculate that  $\frac{\partial}{\partial V_c}  g(W_j(a,b),a,b)\neq 0$ for $j=1,2$.  

Denoting $\mu = L\sqrt{-g(V_c,a,b)}$ as before, recall the definition \er{spfu} of the sparking function:   
\begin{equation}  \label{spfu2} 
D(V_c,a,b,\gamma)= \frac{1}{2}  \left(e^{\mu}  +  e^{-\mu}\right)
+   \frac{V_c}{4 {\mu} }  \left(e^{\mu}  -  e^{-\mu} \right)  -   \frac{\gamma}{1+\gamma}  e^{\frac{V_c}{2}}, 
\end{equation} 
A short calculation shows that if both $g(W,a,b)=0$ and $D(W,a,b,\gamma)=0$, then 
\bqn  \label{Z1} 
\gamma =[1 + \tfrac{W}2]\left[ \exp(\tfrac{W}2) -1 - \tfrac{W}2\right]^{-1}.  
\eqn
The set $Z_1 = \{(a,b,\gamma)\in (\mathbb R_+)^{3}\ ;\ \eqref{Z1} \text{ holds, where } g(W,a,b)=0\}$ obviously has measure zero.  
Thus it is clear that on the complementary set $Z_1^c$ we  have $g(V_c^\dagger(a,b,\gamma),a,b) \ne 0$.  
Within $Z_1^c$ the implicit function theorem ensures that 
the functions $W_j(a,b)$ are continuous  ($j=1,2$).

Clearly the set $\tilde{A}:=A^{\circ}\cap Z_{0}^{c} \cap Z_{1}^{c} $ is open.  
Now let 
$$ Z_2 = \left\{(a,b,\gamma)\in \tilde{A}\ ; \ 
\frac{\pa D}{\pa V_c}(V_c^\dagger,a,b,\gamma)=0 \right\},$$  
where $V_c^\dagger = V_c^\dagger(a,b,\gamma)$.  
We claim that $\tilde{A}\cap Z_2^{c}$ is an open set.  
In order to prove the claim, 
notice that both $g(V_c^{\dagger},a,b)\ne0$ (as shown above) 
and $\frac{\pa D}{\pa V_c}(V_c^\dagger,a,b,\gamma)\ne0$ are true on $\tilde{A} \cap Z_2^c$.  
The sparking function $D(V_c,a,b,\gamma)$ is a real-analytic function of four variables except where $g(V_c,a,b)$ vanishes.  
So for each point $(a,b,\gamma)\in \tilde{A} \cap Z_2^c$, 
we can apply the real-analytic version of the implicit function theorem 
to the equation $D(V_c^\dagger,a,b,\gamma)=0$.
Hence there is a neighborhood of $(a,b,\gamma)$
in which the function $V_c^\dagger$ is real-analytic and
$\frac{\pa D}{\pa V_c}(V_c^\dagger,a,b,\gamma)\ne0$.
Thus $\tilde{A}\cap Z_2^{c}$ is open.  
Furthermore, $V_c^\dagger: \tilde{A}\cap Z_2^c\to\real$ is a real-analytic function for which  
$\frac{\pa D}{\pa V_c}(V_c^\dagger,a,b,\gamma)$ does not vanish.

Next we claim that the $Z_2$ also has $\real^3$-measure zero.
Within $Z_2$ both of the equations, $D=0$ and $ \frac{\pa D}{\pa V_c}=0$, are satisfied by $(V_c^\dagger, a,b,\gamma)$.  
 We calculate 
\begin{equation}  
\frac{\pa D}{\pa V_c}   =   \frac{-L^2g'(V_c)}{2\mu} 
\left\{ (\frac12-\frac{V_c}{4\mu^2}) (e^\mu-e^{-\mu})  +  \frac{V_c}{4\mu}(e^\mu+e^{-\mu})  \right\}
+\frac1{4\mu}(e^\mu-e^{-\mu})   -   \frac12\frac{\gamma}{1+\gamma}  e^{\frac{V_c}{2}}.  
\end{equation} 
The equation $D - 2\frac{\pa D}{\pa V_c}=0$ contains no explicit $\gamma$.  
It is a single equation for $(V_c^\dagger,a,b)$.  
Thus, within $Z_2$, the function $V_c^\dagger$ depends only on $(a,b)$.  
Hence, using \er{spfu2} within $Z_2$, we see that the variable $\gamma$ is determined uniquely by $(a,b)$.    
So, due to the Fubini--Tonelli theorem, $Z_2$ has $\real^3$-measure zero.

Now we define the function 
\begin{align}
F(a,b,\gamma):=&-\gamma e^{\frac{V_c^\dagger}{2}} \psi_e(L) 
\int_0^L \left\{h'(V_c^\dagger/L)-\frac{x}{2}h(V_c^\dagger/L)\right\} 
e^{-\frac{V_c^\dagger}{2L}x}\varphi_e^\dagger(x) \,dx \,   \notag\\
&-Lg'(V_c^\dagger) \int_0^L \psi_e(x) \varphi_e^\dagger(x) \, dx
+\frac{1}{2}\psi_e(L)\left\{\varphi_e^\dagger(L)-L\partial_x\varphi_e^\dagger(L)
-\frac{V_c^\dagger}{2}\varphi_e^\dagger(L) \right\},     
\label{transversality2}
\end{align}
where $\psi_e$ is given in \er{AdjointSol1} and 
$\varphi_e^\dagger$ is equal to the function $S_e$ in \er{Sol2} with \er{positive1}.  
In \er{transversality2} the functions $V_c^\dagger, \psi_e$ and $\varphi_e^\dagger$ 
depend on the parameters $(a,b,\gamma)$.  
Not only is $V_c^\dagger:\tilde{A}\cap Z_2^c\to\real$  real-analytic, but 
we observe from \er{Sol2} and \er{AdjointSol1} 
that $\varphi_e^\dagger$ and $\psi_e$ also depend analytically on $(a,b,\gamma)$.  
It follows that the set  $Z_3 = \{(a,b,\gamma)\in \tilde{A}\ ; \ F(a,b,\gamma)=0\}$ also has measure zero 
because the zero set of any analytic function $\not\equiv0$ must have measure zero.  
In the rest of the proof we will only consider the set $\sA  =  \tilde{A} \cap  Z_2^c \cap Z_3^c 
= A^{\circ}  \cap  Z_0^c \cap Z_1^c  \cap  Z_2^c \cap Z_3^c $.  Because of the definition of $Z_3$, 
we know that $F(a,b,\gamma)\ne0$ within $\sA$.

By differentiating \er{LinearEq1}--\er{LinearEq4} with respect to $\lambda$, we see that
\begin{align}
\partial_\la \pa_{(\rho_i,R_e,V)} \sF_1(\lambda,0,0,0)[1,\varphi_i^\dagger,\varphi_e^\dagger,\varphi_v^\dagger]&=
k_i\partial_x \varphi_i^\dagger
-k_e\left\{h'\left(\lambda\right)e^{-\frac{\lambda}{2}x}-\frac{x}{2}h\left(\lambda\right)e^{-\frac{\lambda}{2}x}\right\} \varphi_e^\dagger,
\lb{LinearEq11}\\
\partial_\la \pa_{(\rho_i,R_e,V)} \sF_2(\lambda,0,0,0)[1,\varphi_i^\dagger,\varphi_e^\dagger,\varphi_v^\dagger]&=
-Lg'(\lambda L)\varphi_e^\dagger,
\lb{LinearEq12}\\
\partial_\la \pa_{(\rho_i,R_e,V)} \sF_3(\lambda,0,0,0)[1,\varphi_i^\dagger,\varphi_e^\dagger,\varphi_v^\dagger]&=
-\frac{x}{2}e^{-\frac{\lambda}{2}x}\varphi_e^\dagger,
\lb{LinearEq13}\\
\partial_\la \pa_{(\rho_i,R_e,V)} \sF_4(\lambda,0,0,0)[1,\varphi_i^\dagger,\varphi_e^\dagger,\varphi_v^\dagger]&=
\frac{1}{2}\varphi_e^\dagger(L)
-\gamma\frac{k_i}{k_e}\left(e^{\frac{\la}{2}L}+\frac{L}{2}\lambda e^{\frac{\la}{2}L}\right)\varphi_i^\dagger(L).
\lb{LinearEq14}
\end{align}
									On the other hand, 
consider the range $R(\partial_{(\rho_i,R_e,V)}\sF(V_c/L,0,0,0))$, which 
is given in \er{reperesent1} and \er{Orthogonal1}.
Owing to these formulas together with \er{domain2}, \er{psi_v}, and \er{AdjointEq3},
the transversality condition \er{transversality1} can be written as
\begin{multline}
\frac{\gamma}{k_e} e^{\frac{V_c^\dagger}{2}} \psi_e(L) \int_0^L \left[ 
k_i \partial_x \varphi_i^\dagger(x)
-k_e\left\{h'(V_c^\dagger/L)-\frac{x}{2}h(V_c^\dagger/L)\right\} 
e^{-\frac{V_c^\dagger}{2L}x}\varphi_e^\dagger(x) 
\right] \,dx \, 
\\
-Lg'(V_c^\dagger) \int_0^L \psi_e(x) \varphi_e^\dagger(x) \, dx
+\psi_e(L)\left\{\frac{1}{2}\varphi_e^\dagger(L)
-\gamma\frac{k_i}{k_e}\left(e^{\frac{V_c^\dagger}{2}} 
+\frac{V_c^\dagger}{2} e^{\frac{V_c^\dagger}{2}}\right)\varphi_i^\dagger(L) \right\}
\neq 0.
\lb{transversality3}
\end{multline}
This is what we have to prove.  
However, the first and last terms in \er{transversality3} add up to  
\begin{align*}
&\frac{\gamma}{k_e} e^{\frac{V_c^\dagger}{2}} \psi_e(L) \int_0^L k_i \partial_x \varphi_i^\dagger(x) \,dx
-\gamma\frac{k_i}{k_e}\left(e^{\frac{V_c^\dagger}{2}}+\frac{V_c^\dagger}{2} e^{\frac{V_c^\dagger}{2}}\right)\psi_e(L)\varphi_i^\dagger(L) 
\\
&=-\gamma\frac{k_i}{k_e}\frac{V_c^\dagger}{2} e^{\frac{V_c^\dagger}{2}}\psi_e(L)\varphi_i^\dagger(L)
=-\frac{1}{2}\psi_e(L)\left\{L\partial_x\varphi_e^\dagger(L)+\frac{V_c^\dagger}{2}\varphi_e^\dagger(L)\right\}. 
\end{align*}
The last equality is due to  \er{LinearEq4} and the fact that 
$(\varphi_i,\varphi_e,\varphi_v) \in N(\partial_{(\rho_i,R_e,V)}\sF(V_c^\dagger/L,0,0,0))$.  
Substituting this simple equality into \er{transversality3} 
shows that the transversality condition \er{transversality2} is precisely the same as 
$F(a,b,c)\ne0$, which we have already shown is true within $\sA$.  
We previously showed that the complement of $\sA$ has measure zero.
\end{proof}

\section{Global Bifurcation}

In this section, we apply a functional-analytic global bifurcation theorem 
to the stationary problem \er{sp0}.  
The theory of global bifurcation goes back to Rabinowitz \cite{Rab}  using topological degree.  
For a nice exposition see \cite{Ki1}.   
 A different version using analytic continuation 
 goes back to Dancer \cite{Da1} with major improvements in \cite{BST1} 
 and a final improvement in \cite{CS1}. 
The specific version that is most convenient to use here is Theorem 6 in \cite{CS1}, 
which is the following:

\begin{thm}[\cite{CS1}]\lb{Global0}
Let $X$ and $Y$ be Banach spaces,
$\sO$ be an open subset of ${\mathbb R}\times X$ and $\sF:\sO \to Y$ be a real-analytic function. 
Suppose that
\begin{enumerate}[{(H}1{)}]
\item $(\lambda,0)\in \sO$ and $\sF(\lambda,0)=0$ for all $\lambda \in \mathbb R$;
\item for some $\lambda^* \in \mathbb R$, $N(\partial_u\sF(\lambda^*,0))$ and 
$Y\backslash R(\partial_u\sF(\lambda^*,0))$ are one-dimensional, 
with the null space generated by $u^*$, which satisfies  the transversality condition
\[
\partial_\la\partial_u\sF(\lambda^*,0)(1,u^*)\notin R(\partial_u\sF(\lambda^*,0)),  
\]
where $\partial_u$ and $\partial_\la\partial_u$ mean Fr\'echet derivatives for $(\la,u) \in \sO$,
and $N(\sL)$ and $R(\sL)$ denote the null space and range of a linear operator $\sL$ between
two Banach spaces;
\item $\partial_u\sF(\lambda,u)$ is a Fredholm operator of index zero for any $(\la,u) \in \sO$
that satisfies the equation $\sF(\lambda,u)=0$;
\item for some sequence $\{\sO_j\}_{j\in \mathbb N}$ of bounded closed subsets of $\sO$ with
$\sO=\cup_{j\in \mathbb N} \sO_j$, 
the set $\{(\la,u) \in \sO ;\sF(\lambda,u)=0\}\cap \sO_j$ is compact for each $j\in\mathbb N$.
\end{enumerate}

Then there exists in $\sO$ a continuous curve 
${\sK}=\{(\la(s),u(s));s \in \mathbb R\}$
of $\sF(\lambda,u)=0$ such that:
\begin{enumerate}[{(C}1{)}]
\item $(\la(0),u(0))=(\la^*,0)$;
\item $u(s)=su^*+o(s)$ in $X$  as $s \to 0$;
\item there exists a neighborhood $\sW$ of $(\la^*,0)$
and $\ve > 0$ sufficiently small such that
\begin{equation*}
\{(\la,u)\in \sW ; u\neq 0 \text{ and } \sF(\la,u)=0 \}
=\{(\la(s),u(s)) ; 0<|s|<\ve\};
\end{equation*}
\item $\sK$ has a real-analytic reparametrization locally around each of its points;
\item one of the following two alternatives occurs:
\begin{enumerate}[(I)]
\item for every $j\in\mathbb N$, there exists $s_j>0$ such that $(\la(s),u(s))\notin \sO_j$
for all $s \in \mathbb R$ with $|s|>s_j$;
\item there exists $T>0$ such that
$(\la(s),u(s))=(\la(s+T),u(s+T))$ for all $s \in \mathbb R$.
\end{enumerate}
\end{enumerate}
Moreover, such a curve of solutions of $\sF(\lambda,u)=0$
having the properties (C1)-(C5) is unique (up to reparametrization).
\end{thm}

Hypothesis $(H2)$ is the same local bifurcation condition as in Theorem \ref{Local}, 
while $(H3)$ and $(H4)$ are the global ones.  
$(C1)-(C3)$ are local conclusions, $(C4)$ is a statement of regularity, 
which is a consequence of the real-analyticity of $\sF$. 
$(C5)$ is the global conclusion which states that {\it either} the curve 
reaches the boundary of the set $\sO_j$ {\it or} the curve is periodic 
(that is, forms a closed loop).  
The hypotheses $(H3)$ and $(H4)$ are validated in Lemmas \ref{index1} and \ref{cpt1},
respectively.
For that purpose, consider the linearized operator around an arbitrary triple of functions 
$(\rho_i^0,R_e^0,V^0)\in X$.  
\begin{lem}\lb{index1} 
For any $(\lambda,\rho_i^0,R_e^0,V^0)\in {\sO}$, 
the Fr\'echet derivative 
$\sL^0 = \partial_{(\rho_i,R_e,V)} \sF(\lambda,\rho_i^0,R_e^0,V^0)$
is a linear Fredholm operator of index zero from $X$ to $Y$.
\end{lem}
\begin{proof}
For any fixed choice of $(\lambda,\rho_i^0,R_e^0,V^0)$, 
we know that $\inf_x \partial_x V^0 + \lambda > 0$.   
The operator $\sL^0=(\sL_1,\sL_2,\sL_3,\sL_4)$ acting linearly on the triple $(S_i,S_e,W)\in X$  has the form 
\begin{align}
\sL_1 = \sL_1(S_i,S_e,W) =& 
k_i \partial_x(\{\pa_x V^0 + \lambda\} S_i) + b_1\pa_x^2 W + b_2S_i + b_3S_e + b_4\pa_x W , 
\lb{sL1}\\
\sL_2 = \sL_2(S_i,S_e,W) =& 
-\pa_x^2 S_e  + a_1\pa_xS_e + b_5 S_e + b_6\pa_x^2W + b_7\pa_xW  , 
\lb{sL2}\\
\sL_3 = \sL_3(S_i,S_e,W) =&
-\pa_x ^2 W  + a_2S_i + a_3S_e ,
\lb{sL3} \\
\sL_4 = \sL_3(S_i,S_e,W) =&
\pa_x S_e(L) + (\pa_xV^0(L)+\tfrac\lambda2)S_e(L)  +  \pa_xW(L) R_e^0(L) 
\notag \\
&-\tfrac{\gamma k_i}{k_e} \exp(\tfrac{\lambda}{2}L) [ \pa_xV^0(L)+\lambda)S_i(L)  
+  \pa_xW(L)\rho_i^0(L)] , 
\lb{sL3}
\end{align}
where the coefficients $a_1=-\pa_xV^0   $, $a_2$ and $a_3$ belong to $C^1([0,L])$
and the coefficients $b_1,...,b_7$ belong to $C^0([0,L])$.

Let us first show that the linear operator $\sL^0$ has a finite-dimensional nullspace and a closed range.
By \cite[Theorem 12.12]{Wlo} or \cite[Exercise 6.9.1]{Brz}, 
it is equivalent to prove that $\sL^0$ satisfies the estimate 
\bqn  \label{ellest} 
C \|(S_i,S_e,W)\|_X  \le  \|\sL^0(S_i,S_e,W)\|_Y  +  \|(S_i,S_e,W)\|_Z 
\eqn
for all $(S_i,S_e,W)\in X$ and 
for some constant $C$ depending only on $(\lambda,\rho_i^0,R_e^0,V^0)$,
where 
\[
Z:= C^0([0,L]) \times C^0([0,L]) \times C^1([0,L]).
\]
Keeping in mind that  $\pa_x V^0 + \lambda\ge 1/j$, we see from \er{sL1} and \er{sL3} that  
$S_i$ can be estimated by 
 \begin{align}
 \|\pa_x S_i\|_{C^0} 
 &=\|(\pa_x V^0 + \lambda)^{-1}
 \left(\{\pa_x(\pa_x V^0 + \lambda)\} S_i
 +b_1\pa_x^2 W + b_2S_i + b_3S_e + b_4\pa_x W-\sL_1\right)\|_{C^0}
 \notag\\
 &\leq C(\|S_i\|_{C^0}+\|S_e\|_{C^0}+\|W\|_{C^2}+\|\sL_1\|_{C^0})
 \notag\\
 &\leq  C\|\sL^0(S_i,S_e,W)\|_Y  +  C\|(S_i,S_e,W)\|_Z.
 \lb{estsL1}
 \end{align}
 Next, \er{sL3} leads to the required estimate of $W$ as follows: 
 \begin{align}
 \|\pa_x^2 W\|_{C^1} 
 =\|a_2S_i + a_3S_e-\sL_3\|_{C^1}
 \leq  C\|\sL^0(S_i,S_e,W)\|_Y  +  C\|(S_i,S_e,W)\|_Z.
 \lb{estsL3}
 \end{align}
 We also have $\|\pa_x W\|_{C^0} \le L\|\pa_x^2W\|_{C^0}$ 
 because $\int_0^L \pa_xW(x)dx=0$. 
 
 Finally, we estimate $S_e$ as follows.  
 Due to the bounds on $S_i$ and $W$, the equation \er{sL2} 
 implies that $\pa_x^2S_e + (\pa_x V^0)\partial_xS_e$ is bounded by the right side of \eqref{ellest}.  
 Furthermore, $S_e(0)=0$ and $\pa_xS_e(L) + (\pa_xV^0(L)+\tfrac\lambda2) S_e(L)$ is also bounded.  
Thus $\pa_x\{\pa_xS_e + (\pa_x V^0)S_e \}$   is also bounded.   
Integrating from $x$ to $L$, we find that 
\[   \pa_xS_e(x) + \pa_x V^0(x)S_e(x) -  \pa_xS_e(L) + \pa_x V^0(L)S_e(L)   \]  
is also bounded,  whence $\pa_xS_e(x)$ is bounded as well.  
The preceding estimates on $S_i, W$ and $S_e$ prove \er{ellest}. 

 Owing to the fact $\lim_{V_c\to 0}D(V_c)>0$, we can find 
 a constant $V_c'>0$ such that $D(V_c')>0$.
 The preceding lemmas state that 
 the nullspace of $\partial_{(\rho_i,R_e,V)} \sF(V_c'/L,0,0,0)$ has dimension zero and 
 the codimension of its range is also zero,
 so that its index is zero.    Because $\sO$ is connected 
 and the index is a topological invariant \cite[Theorem 4.51, p166]{AA1}, 
 $\sL^0$ also has index zero.  
 This means that the codimension of $\sL^0$ is also finite.
 This completes the proof of Lemma \ref{index1}.
\end{proof}

\begin{lem}\lb{cpt1}
For each $j\in \mathbb N$, the set 
$K_j  =  \{(\lambda,\rho_i,R_e,V)\in {\sO_j};\ \sF(\lambda,\rho_i,R_e,V)=0\}$
is  compact in ${\mathbb R}\times X$.
\end{lem}
\begin{proof}
Let  $\{(\lambda_n,\rho_{in},R_{en},V_n)\}$ be any sequence in  $K_j$. 
It suffices to show that  it has a convergent subsequence
whose limit also belongs to $K_j$.  
By the assumed bound 
$|\lambda_n|+\|(\rho_{in},R_{en},V_n)\|_{X}\leq j$, 
there exists a subsequence, still denoted by $\{(\lambda_n,\rho_{in},R_{en},V_n)\}$,
and $(\lambda,\rho_i,R_e,V)$ such that
\begin{equation}\lb{converge1}
\left\{
 \begin{array}{lllll}
 \lambda_{n} & \to & \lambda & \text{in} & \mathbb R,
 \\
 \rho_{in} & \to & \rho_i & \text{in} & C^0([0,L]),
 \\
 R_{en} & \to & R_{e} & \text{in} & C^1([0,L]),
 \\
 V_{n} & \to & V & \text{in} & C^2([0,L]).
 \end{array}
\right.
\end{equation}
Furthermore, 
\[
\partial_x V+\lambda\geq \tfrac1j.
\]
Since $\sO_j$ is closed in $X$,
it remains to show that 
\begin{gather*}
\sF_j(\lambda,\rho_i,R_e,V)=0 \qu \text{for $j=1,2,3,4,$}
\\
\rho_{in}\to \rho_i \ \text{in} \ C^1([0,L]), \qu
R_{en}\to R_e \ \text{in} \ C^2([0,L]), \qu
V_n\to V \ \text{in} \ C^3([0,L]).
\end{gather*}

\medskip

\noindent

Now the first equation $\sF_1(\lambda_n,\rho_{in},R_{en},V_n)=0$ with $\rho_{in}(0)=0$ is equivalent to
\[
 \rho_{in}(x)=\frac{k_e}{k_i}(\partial_x V_n(x)+\lambda_n)^{-1}
\int_0^x h(\partial_x V_n(y)+\lambda_n)e^{-\frac{\lambda_n}{2}y}R_{en}(y)\,dy.
\]
Taking the limit and using \er{converge1}, we see that 
\[
 \rho_i(x)=\frac{k_e}{k_i}(\partial_x V(x)+\lambda)^{-1}
\int_0^x h(\partial_x V(y)+\lambda)e^{-\frac{\lambda}{2}y}R_{e}(y)\,dy, 
\]
where the right hand side converges in $C^1([0,L])$.
Hence, we see that $\sF_1(\lambda,\rho_i,R_e,V)=0$ and $\rho_{in}\to \rho_i$ in $C^1([0,L])$.

Taking the limit using \eqref{converge1} in the third equation $\sF_3(\lambda_n,\rho_{in},R_{en},V_n)=0$ immediately leads to 
\[
 \partial_x^2 V=\rho_i-e^{-\frac{\lambda}{2}x}R_{e}.
\]
Hence $\sF_3(\lambda,\rho_i,R_e,V)=0$ and $V_n\to V$ in $C^3([0,L])$.

The second equation $\sF_2(\lambda_n,\rho_{in},R_{en},V_n)=0$ 
can be written as 
\[   \pa_x \{\pa_x R_{en} - (\pa_xV_n)R_{en}\}  = \{\tfrac{\la_n}2 +\tfrac{{\la_n}^2}{4} - h(\pa_xV_n+\la_n) \} R_{en}. 
\]
Because the right side converges in $C^1([0,L])$, we see that 
$\{\pa_x R_{en} - (\pa_xV_n)R_{en}\}$ converges in $C^2([0,L])$. 
But $(\pa_xV_n)R_{en}$ converges in $C^1([0,L])$.  
Hence $\pa_xR_{en}$ converges in $C^1([0,L])$, 
which means that $R_{en}$ converges to $R$ in $C^2([0,L])$.

It is obvious from \er{converge1} and $\sF_4(\lambda_n,\rho_{in},R_{en},V_n)=0$
that $\sF_4(\lambda,\rho_i,R_e,V)=0$ holds.
\end{proof}

As we have checked all conditions in Theorem \ref{Global0}, 
the following conclusion is valid.
\begin{thm}\lb{Global1}
Assume that the sparking voltage $V_c^\dagger$, defined by \eqref{bp1}, exists.
There exists in the open set $\sO$ a continuous curve 
${\sK}=\{(\la(s),\rho_i(s),R_e(s),V(s));s \in \mathbb R\}  \subset  \mathbb R\times X$
of stationary solutions to problem \er{sp0} such that
\begin{enumerate}[{(C}1{)}]
\item $(\la(0),\rho_i(0),R_e(0),V(0))=(V_c^\dagger/L,0,0,0)$, where $V_c^\dagger$ is defined in \er{bp1};
\item $(\rho_i(s),R_e(s),V(s))=s(\varphi_i^\dagger,\varphi_e^\dagger,\varphi_v^\dagger)+o(s)$ in the space $X$
 as $s \to 0$, where $(\varphi_i^\dagger,\varphi_e^\dagger,\varphi_v^\dagger)$ is
a basis with \er{positive0} of $N(\partial_{(\rho_i,R_e,V)}\sF(V_c^\dagger/L,0,0,0))$.
\item there exists a neighborhood $\sW$ of $(V_c^\dagger/L,0,0,0)$
and $\ve < 1$ such that
\begin{multline*}
\{(\la,\rho_i,R_e,V)\in \sW ; (\rho_i,R_e,V)\neq (0,0,0), \ \sF(\la,\rho_i,R_e,V)=0 \}
\\
=\{(\la(s),\rho_i(s),R_e(s),V(s)) ; 0<|s|<\ve\};
\end{multline*}
\item $\sK$ has a real-analytic reparametrization locally around each of its points;
\item at least one of the following four alternatives occurs:
\begin{enumerate}[(a)]
\item $\varliminf_{s \to \infty}\la(s)= 0$;
\item $\varliminf_{s \to \infty}(\inf_{x \in I}\partial_x V(x,s)+\la(s))=0$;
\item $\varlimsup_{s \to \infty}(\|\rho_i\|_{C^1}+\|R_e\|_{C^2}+\|V\|_{C^3}+\lambda)(s)=\infty$;
\item there exists $T>0$ such that
$$(\la(s),\rho_i(s),R_e(s),V(s))=(\la(s+T),\rho_i(s+T),R_e(s+T),V(s+T))$$ for all $s \in \mathbb R$.
\end{enumerate}
\end{enumerate}
Moreover, such a curve of solutions to problem \er{sp0} having the properties (C1)-(C5)
is unique (up to reparametrization).
\end{thm}

Conditions (C1)-(C3) are an expression of the local bifurcation, while (C4)-(C5) are assertions about the 
global curve $\sK$.  Alternative (c)  asserts that $\sK$ may be unbounded.  
Alternative (d) asserts that $\sK$ may form a closed curve (a `loop').

\section{Positive Densities}
Of course, we should keep in mind that for the physical problem  $\rho_i$ and $R_e$ are densities 
of particles and so they should be non-negative. 
In this section we investigate the part of the curve $\sK$ that corresponds to such densities. 
We will often suppress the variable $x$, as in $\rho_i(s)=\rho_i(s,\cdot), R_e(s)=R_e(s,\cdot), V(s)=V(s,\cdot)$.  

A basic observation is the following theorem, which states that {\it either (i)} 
$\rho_i$ and $R_e$ remain positive {\it or (ii)} the curve of positive solutions 
forms a half-loop going from $V_c^\dagger$ to some other voltage $V_c^\ddagger$.
Here $V_c^\dagger$ is defined in \er{bp1} 
and $V_c^\ddagger$ is a voltage with \er{positive1} and $V_c^\dagger<V_c^\ddagger$.
We remark that the curve $\sK$ is never the half-loop 
unless a voltage $V_c^\ddagger > V_c^\dagger$ exists satisfying \er{positive1}.

\begin{thm}\lb{Global2}
Assume the sparking voltage $V_c^\dagger$ exists.  
For the global bifurcation curve $\sK = (\la(s),\rho_i(s),R_e(s),V(s))$ in Theorem \ref{Global1},
one of the following two alternatives occurs.
\begin{enumerate}[(i)] 
\item  $\rho_i(s,x)>0$ and $R_e(s,x)>0$ \  for all $0<s<\infty$ and $x \in (0,L]$.
\item there exists a voltage $V_c^\ddagger$ satisfying  \er{positive1} and $V_c^\dagger<V_c^\ddagger$  
and a finite parameter value $s^\ddagger>0$ such that 
\begin{enumerate}[(1)]
\item $\rho_i(s,x)>0$ and $R_e(s,x)>0$ \ for all $s\in(0,s^\ddagger)$ and $x \in (0,L]$;
\item $(\la(s^\ddagger),\rho_i(s^\ddagger),R_e(s^\ddagger),V(s^\ddagger))=(V_{c}^\ddagger/L,0,0,0)$;
\item $(\rho_i(s),R_e(s))=(s^\ddagger-s)(\varphi_i^\ddagger,\varphi_e^\ddagger)+o(|s-s^\ddagger|)$
as $s\nearrow s^\ddagger$, where $(\varphi_i^\ddagger,\varphi_e^\ddagger)$ 
is a basis with \er{positive0} of $N(\partial_{(\rho_i,R_e,V)}\sF(V_c^\ddagger/L,0,0,0))$;
\item  $\rho_i(s,x)<0$ and $R_e(s,x)<0$ \ for $0<s-s^\ddagger\ll1$ and $x \in (0,L]$.
\end{enumerate}
\end{enumerate}
\end{thm}
\begin{proof}
First let us define  
\begin{equation}   \label{Tddagger}
 s^\ddagger:=\inf\{s>0 : R_e(s,x_0)=0 \ \text{for some $x_0 \in (0,L]$}\}.
\end{equation}    
Clearly $R_e>0$ in $(0,s^\ddagger) \times (0,L]$.     
By (C2) in Theorem \ref{Global1}, $s^\ddagger>0$. 
				If $s^\ddagger=\infty$, 
then $R_e>0$ in $(0,\infty) \times (0,L]$.
Also  $\partial_x V+\lambda$ is positive owing to 
$(\la(s),\rho_i(s),R_e(s),V(s)) \in \sO$. 
Then the following formula from \eqref{re1}
also yields $\rho_i>0$.
\begin{equation}\label{R_i1}
\rho_i(x)=\frac{k_e}{k_i}(\partial_x V(x)+\lambda)^{-1}
\int_0^x h(\partial_x V(y)+\lambda)e^{-\frac{\lambda}{2}y}R_{e}(y)\,dy. 
\end{equation}
Thus alternative {\it (i)} is valid.  

Assuming that $s^\ddagger<\infty$, we will show that {\it (ii)} happens.  
First we will show that $R_e(s^\ddagger,\cdot)$ vanishes identically.  
Certainly 
$R_e(s^\ddagger,\cdot)$ takes the value zero, which is its minimum, 
at some point $x_0\in \bar{I}=[0,L]$. 
In  case $x_0 \in I$, $\partial_x R_e(s^\ddagger,x_0)=0$ also holds.
Solving $\sF_2(\lambda,\rho_i,R_e,V)=0$ with $R_e(s^\ddagger,x_0)=\partial_x R_e(s^\ddagger,x_0)=0$,
we see by uniqueness that $R_e(s^\ddagger)\equiv 0$.  
Secondly, in case $x_0=0$, by \er{Tddagger} there exists a sequence $\{(s_n,x_n)\}_{n \in \mathbb N}$ such that 
$R_e(s_n,x_n)=0$ with $s_n \searrow s^\ddagger$ and $x_n \searrow 0$.
Rolle's theorem ensures that there also exists some $y_n\in(0,x_n)$ 
such that $\partial_x R_e(s_n,y_n)=0$.
Letting $n\to \infty$, we see that $y_n\to 0$ and thus $\partial_x R_e(s^\ddagger,0)=0$.
Hence we again deduce  by uniqueness that $R_e\equiv 0$.  
Thirdly, in case $x_0=L$, it is obvious that $\partial_x R_e(s^\ddagger,L)\leq 0$.
On the other hand, we see from $\sF_4=0$ and \er{R_i1} that
\[
\partial_xR_e(s^\ddagger,L)=\gamma\frac{k_i}{k_e}e^{\frac{\la}{2}L}
\left(\partial_x V(L)+{\la}\right)\ro_i(s^\ddagger,L)\geq 0.  
\]
This leads to $\partial_x R_e(s^\ddagger,L)=0$ so that $R_e\equiv 0$ once again.  
Therefore we conclude that $R_e\equiv 0$ in every case.  
By \er{R_i1}, we also have $\rho_i\equiv 0$ and thus $V\equiv0$.
Hence $(\rho_i,R_e,V)(s^\ddagger)=(0,0,0)$ is the trivial solution.
So (1) and (2) in the theorem are valid. 

Continuing to assume that $s^\ddagger<\infty$, we now know that $\rho_i$, $R_e$ and $V$ are 
identically zero at $s=s^\ddagger$.  
We define $V_c^\ddagger = L\ \la(s^\ddagger)$.  
By the simple bifurcation theorem of \cite{CR1}, the nullspace 
$\sN = N[\partial_{(\rho_i,R_e,V)} \sF(\la(s^\ddagger),0,0,0)]$ is non-trivial 
because the curve $\sK$ crosses the trivial curve transversely at $s=s^\ddagger$.   
So by Lemma \ref{null1}, we have $D(V_c^\ddagger)=0$.   
It remains to prove (3) and (4)  and also that $V_c^\ddagger > V_c^\dagger$ 
and  $g(V_c^\ddagger) \le \frac{\pi^2}{L^2}$ .

Suppose on the contrary that $g(V_c^\ddagger) > \frac{\pi^2}{L^2}$.   
Then as in the proof of Lemma \ref{null1}, the nullspace $\sN$ has a basis $(\varphi_i,\varphi_e,\varphi_v)$ 
with 
\[
{\varphi}_e(x) = \sin\sqrt{g({V}_c^\ddagger)} x,
\quad
\sqrt{g({V}_c^\ddagger)} >\frac{\pi}{L}.
\]
In that case the function $\varphi_e$ has a node (changes its sign) in the interval $I$.  
Therefore $R_e(s,\cdot)$ also has a node for $s$ near  $s^\ddagger$, Theorem
which contradicts the positivity.  
Thus $g(V_c^\ddagger) \le \frac{\pi^2}{L^2}$ so that the basis of $\sN$ is positive,  
due to  Lemma \ref{null1}.   Thus (3) and (4) are valid.  

Finally, suppose that $V_c^\ddagger = V_c^\dagger$.  
Then $\la(s^\ddagger) = {V_c^\dagger}/L$, 
so that the curve $\sK$ goes from the point $P=({V_c^\dagger}/L,0,0,0)$ at $s=0$ 
to the same point $P$ at $s={s^\ddagger}$.   
By (C3) and (C4) of Theorem \ref{Global1},   
$\sK$ is a simple curve at $P$ and is real-analytic. 
So the only way $\sK$ could go from $P$ to $P$ would be if it were a loop with the part with 
$s$ approaching $s^\ddagger$ from below coinciding 
with the part with $s$ approaching $0$ from below ($s<0$).  
By (C2) of Theorem \ref{Global1}, $\rho_i(s,\cdot)$ and $R_e(s,\cdot)$ would be negative 
for $-1\ll s-s^\ddagger<0$, which would contradict their positivity.  
Hence $V_c^\ddagger >V_c^\dagger$.
\end{proof}

Since $\rho_i$ and $R_e e^{-V_cx/2L}$ are the densities of the ions and electrons, respectively, 
we are interested only in the positive solutions.  
Let us investigate in detail the case 
that the global positivity alternative (i) in Theorem \ref{Global2} occurs.
More precisely, the next three lemmas show that 
if any one of the alternatives (a) or (b) in Theorem \ref{Global1} occurs,
then alternative (c) also occurs.
In these proofs, we use the written boundary condition from \er{rb2} and \er{R_i1}: 
\begin{equation}\lb{bc0}
\partial_xR_e(L)=-\left(\partial_x V(L)+\frac{\la}{2}\right)R_e(L)
+\gamma e^{\frac{\la}{2}L}
\int_0^L h\left(\partial_x V(x)+\lambda\right)e^{-\frac{\lambda}{2}x}R_e(x) \,dx.
\end{equation}
and the elementary Poincar\'e inequality
\begin{equation}\lb{Poincare}
\|u\|_{L^2} \leq \sqrt L \|\partial_x u\|_{L^2}
\quad \text{for $u \in \{f \in H^1(I); \ f(0)=0\}$}.
\end{equation}

\begin{lem}\lb{lem(a)}
Assume alternative (i) in Theorem \ref{Global2}.
If $\varliminf_{s \to \infty}\la(s)= 0$, then 
$\sup_{s>0}\|V(s)\|_{C^2}$ is unbounded.
\end{lem}

\begin{proof}
On the contrary suppose that $\sup_{s>0}\|V(s)\|_{C^2}$ is bounded.
Because $\varliminf_{s \to \infty}\la(s)= 0$ and $(\partial_x V+\la)(s,x)>0$,
there exists a sequence $\{s_n\}_{n \in \mathbb N}$ 
and limits $(\la^*,V^*)$ such that 
\begin{gather}
\left\{
\begin{array}{lllll}
 \lambda(s_n) & \to & 0 & \text{in} & \mathbb R,
 \\
 V(s_n) & \to & V^* & \text{in} & C^1([0,L]),
 \end{array}\right.
\lb{converge4}\\
V^*(0)=V^*(L)=0, 
\lb{boundary2}\\
\partial_x V^*\geq 0.
\lb{p2}
\end{gather}
The boundary condition \er{boundary2} means that $\int_0^L \partial_x V^*(x)\,dx =0$.
This together with \er{p2} implies $\partial_x V^* \equiv 0$.
Using \er{boundary2} again, we have $V^* \equiv 0$.

It follows that
for suitably large $n$ the three  expressions $\|h(\partial_xV(s_n)+\la(s_n))\|_{C^0}$, $|\la(s_n)|$ and 
$\|V(s_n)\|_{C^2}$, 
are arbitrarily small.  
Multiplying $\sF_2(\la(s_n),\rho_i(s_n),R_e(s_n),V(s_n))=0$ by $R_e(s_n)$ leads to
\begin{align*}
(\partial_x R_e)^2(s_n)&=\partial_x\left\{R_e(s_n) \partial_x R_e(s_n) +\partial_x V(s_n)R_e^2(s_n) \right\}
-3\partial_x V(s_n)R_e(s_n)\partial_xR_e(s_n)
\\
&\quad +\left\{
\frac{\lambda(s_n)}{2}\partial_x V(s_n)
+\frac{\lambda^2(s_n)}{4}
-h\left(\partial_x V(s_n)+\lambda(s_n)\right)\right\}R_e^2(s_n).
\end{align*}
Then integrating this by parts over $[0,L]$, using $R_e(s_n,0)=0$, 
and rewriting $\partial_x R_e(s_n,L)$ by \er{bc0}, we  have 
\begin{align*}
\int_0^L (\partial_x R_e)^2(s_n) \,dx
&=-\left(\partial_x V(s_n,L)+\frac{\la(s_n)}{2}\right)R_e^2(s_n,L)
\\
&\quad +\gamma e^{\frac{\la(s_n)}{2}L}R_e(s_n,L)
\int_0^L h\left(\partial_x V(s_n)+\lambda(s_n)\right)e^{-\frac{\lambda(s_n)}{2}x}R_e(s_n) \,dx
\\
&\quad -3\int_0^L
\partial_x V(s_n)R_e(s_n)\partial_xR_e(s_n)\,dx 
+\partial_x V(s_n,L)R_e^2(s_n,L)
\\
&\quad -\int_0^L\left\{
\frac{\lambda(s_n)}{2}\partial_x V(s_n)
+\frac{\lambda^2(s_n)}{4}
-h\left(\partial_x V(s_n)+\lambda(s_n)\right)\right\}R_e^2(s_n)\,dx 
\\
&\leq \frac{1}{2}\int_0^L (\partial_x R_e)^2(s_n) \,dx,
\end{align*}
where we also have used 
Sobolev's and Poincar\'e's inequalities 
and taken $n$ suitably large in deriving the last inequality.
Hence $\partial_xR_e(s_n) \equiv0$.  
Since $R_e$ vanishes at $x=0$, we conclude that $R_e(s_n)\equiv 0$,   
which contradicts the assumed positivity. 
\end{proof}

\begin{lem}\lb{lem(b)}
Assume alternative (i) in Theorem \ref{Global2}.
If $\varliminf_{s \to \infty}\{\inf_{x \in I}\partial_x V(s,x)+\la(s)\}=0$, 
then $\sup_{s>0}\{\|\rho_i(s)\|_{C^0}+\|R_e(s)\|_{C^2}+\|V(s)\|_{C^2}+\lambda(s)\}$ is unbounded.
\end{lem}
\begin{proof}
On the contrary, suppose that $\sup_{s>0}\{\|\rho_i(s)\|_{C^0}+\|R_e(s)\|_{C^2}+\|V(s)\|_{C^2}+\lambda(s)\}$ is bounded.
We see from $\varliminf_{s \to \infty}\{\inf_{x \in I}\partial_x V(s,x)+\la(s)\}=0$ 
that there exist a sequence $\{s_n\}_{n \in \mathbb N}$
and a quadruple $(\la^*,\rho_i^*,R_e^*,V^*)$ with $\la^*<\infty $ such that
\begin{gather}
\left\{
\begin{array}{llllll}
 \lambda(s_n) & \to & \la^* & \text{in} & \mathbb R,
 \\
 \rho_{i}(s_n) & \rightharpoonup & \rho_i^* & \text{in} & L^\infty(0,L) & \text{weakly-star},
 \\
 R_{e}(s_n) & \to & R_{e}^* & \text{in} & C^1([0,L]),
 \\
 \partial_x^2 R_{e}(s_n) & \rightharpoonup & \partial_x^2R_{e}^* & \text{in} & L^\infty(0,L) & \text{weakly-star},
 \\
 V(s_n) & \to & V^* & \text{in} & C^1([0,L]),
 \\
 \partial_x^2 V(s_n) & \rightharpoonup & \partial_x^2V^* & \text{in} & L^\infty(0,L) & \text{weakly-star},
 \end{array}\right.
\lb{converge3}\\
R_{e}^*(0)=V^*(0)=V^*(L)=0,
\lb{boundary1} \\
\rho_i^* \geq 0, \quad R_e^* \geq 0,
\lb{p1}\\
\inf_{x\in[0,L]} (\partial_x V^*+\la^*)(x)=0.
\lb{singular1}
\end{gather}
			We shall  show that 
\begin{equation*}
\sF_j(\lambda^*,\rho_i^*,R_e^*,V^*)=0 \quad \text{for a.e. $x$ and  $j=1,2,3$.}
\end{equation*}
The equation $\sF_1(\lambda(s_n),\rho_{i}(s_n),R_{e}(s_n),V(s_n))=0$ 
with $\rho_{i}(s_n,0)=0$ is equivalent to
\begin{equation*}
(\partial_x V(s_n)+\lambda(s_n))\rho_{i}(s_n)=\frac{k_e}{k_i}
\int_0^x h(\partial_x V(s_n)+\lambda(s_n))e^{-\frac{\lambda(s)}{2}y}R_{e}(s_n)\,dy. 
\end{equation*}
		Multiplying  
by a test function $\varphi \in C^0([0,L])$ and integrating over $[0,L]$, we obtain
\begin{equation}\lb{rho0}
\int_0^L (\partial_x V(s_n)+\lambda(s_n))\rho_{i}(s_n)\varphi \,dx
=\int_0^L \frac{k_e}{k_i}\left(
\int_0^x h(\partial_x V(s_n)+\lambda(s_n))e^{-\frac{\lambda(s_n)}{2}y}R_{e}(s_n)\,dy 
\right)\varphi\,dx. 
\end{equation}  
		We note that 
\begin{align*}
&{}
\left|\int_0^L \left\{(\partial_x V(s_n)+\lambda(s_n))\rho_{i}(s_n)
-(\partial_x V^*+\lambda^*)\rho_{i}^*\right\}\varphi \,dx\right|
\\
&\leq 
\left|\int_0^L \left\{\partial_x V(s_n)+\lambda(s_n)
-\partial_x V^*-\lambda^* \right\} \rho_{i}(s_n) \varphi \,dx\right|
+\left|\int_0^L (\rho_i(s_n) - \rho_i^*){(\partial_x V^*+\lambda^*)\varphi} \,dx\right|.
\end{align*}
		So passing to the limit $n\to\infty$ in \er{rho0} and using \er{converge3}, we obtain 
\begin{equation*}
\int_0^L (\partial_x V^*+\lambda^*)\rho_{i}^*\varphi \,dx
=\int_0^L \frac{k_e}{k_i}\left(
\int_0^x h(\partial_x V^*+\lambda^*)e^{-\frac{\lambda^*}{2}y}R_{e}^*\,dy 
\right)\varphi\,dx
\quad \text{for any $\varphi \in C^0([0,L])$}.
\end{equation*}
This immediately gives 
\begin{equation}\lb{rho1}
(\partial_x V^*+\lambda^*)\rho_{i}^*
=\frac{k_e}{k_i}
\int_0^x h(\partial_x V^*+\lambda^*)e^{-\frac{\lambda^*}{2}y}R_{e}^*\,dy \quad a.e.,
\end{equation}
which is equivalent to $\sF_1(\lambda^*,\rho_i^*,R_e^*,V^*)=0$ a.e.

We can write  $\sF_2(\lambda(s_n),\rho_{i}(s_n),R_{e}(s_n),V(s_n))=0$ and
$R_{e}(s_n,0)=0$ weakly as 
\[
\int_0^L \partial_x R_{e}(s_n)\partial_x \varphi \,dx
+\frac{(\lambda(s_n))^2}{4}\int_0^L R_{e}(s_n)\varphi \,dx
=-\int_0^L G_{2n}\varphi \,dx \qu \text{for any $\varphi \in H^1_0(0,L)$},
\]
where 
\[
G_{2n}:= -\partial_x V(s_n) \partial_x R_{e}(s_n)
+\left\{\frac{\lambda(s_n)}{2}\partial_x V(s_n)
-\partial_x^2V(s_n)
-h\left(\partial_x V(s_n)+\lambda(s_n)\right)\right\}R_{e}(s_n).
\]
Noting that 
\begin{align*}
&{}
\left|\int_0^L \{\partial_x^2 V(s_n) R_{e}(s_n)- (\partial_x^2 V^*) R_{e}^*\}\varphi \,dx\right|
\\
&\leq 
\left|\int_0^L \partial_x^2 V(s_n)(R_{e}(s_n)-R_{e}^*) \varphi \,dx\right|
+\left|\int_0^L (\partial_x^2 V(s_n) - \partial_x^2 V^*) {R_{e}^*\varphi} \,dx\right|, 
\end{align*}
taking the limit $n\to\infty$ in the weak form, and using \er{converge3}, we have
\[
\int_0^L (\partial_x R_{e}^*)(\partial_x \varphi) \,dx
+\frac{\lambda^2}{4}\int_0^L R_{e}^*\varphi \,dx
=-\int_0^L G_{2}^*\varphi \,dx \qu \text{for any $\varphi \in H^1_0(0,L)$},
\]
where 
\[
G_{2}^*:= -(\partial_x V^*) \partial_x R_{e}^*
+\left\{\frac{\lambda}{2}\partial_x V^*
-\partial_x^2V^*
-h\left(\partial_x V^*+\lambda^*\right)\right\}R_{e}^* \in L^2(0,L).
\]
This and \er{converge3} mean that $R_e^* \in C^1([0,L])\cap W^{2,\infty}(0,L)$ 
satisfies $\sF_2=0$.
Similarly we can show $\sF_3(\lambda^*,\rho_i^*,R_e^*,V^*)=0$.

We now set
\[
 x_*:=\inf\{x \in [0,L];(\partial_x V^*+\lambda^*)(x)=0\}.
\]
We divide our proof into two cases $x_* = 0$ and $x_* >0$. 

\smallskip

We first consider the case $x_* >0$. 
The equation \er{rho1}, which holds for a sequence $x_\nu\to x_*$, yields the inequality  
\begin{equation*}
0=(\partial_x V^*+\lambda^*) \|\rho_i\|_{L^\infty(I)}  \ge \frac{k_e}{k_i}
\int_0^{x_*} h(\partial_x V^*+\lambda^*)e^{-\frac{\lambda^*}{2}y}R_{e}^*\,dy. 
\end{equation*}
Together with the nonnegativity \er{p1} this implies that 
$(h(\partial_x V^*+\lambda^*)e^{-\frac{\lambda^*}{2}\cdot}R_{e}^*)(x)=0$ 
for $x \in [0,x_*]$.
From the definition of $x_*$, 
we see that 
\begin{equation}\lb{regular1}
(\partial_x V^*+\lambda^*)(x)>0 \quad \text{for $x\in [0,x_*)$,}
\end{equation} 
so that $h(\partial_x V^*+\lambda^*)>0$ on $[0,x_*)$.
Therefore, $R_e^*(x)\equiv0$  in $[0,x_*)$.
Hence from \er{rho1} and \er{regular1},  
$\rho_i^*=0$ a.e.  in $[0,x_*)$.
Now from the equation $\sF_3(\lambda^*,\rho_i^*,R_e^*,V^*)=0$ 
we see that $\partial_x V^*$ is a constant in $(0,x_*)$.  Thus 
$\partial_x V^*+\lambda^*=0$ in $[0,x_*]$.
This contradicts the definition of $x_*$.

\smallskip

Now consider the other case  $x_* = 0$.  
We first suppose that there exists $y_0>0$ such that
$(\partial_x V^*+\lambda^*)(y_0)>0$. Let us set
\[
 y^*:=\sup\{x<y_0 ; (\partial_x V^*+\lambda^*)(x)=0 \}.
\]
Note that $y^* \in [0,y_0)$ and $(\partial_x V^*+\lambda^*)(y^*)=0$.
On the other hand, integrating $\sF_1(\lambda^*,\rho_i^*,R_e^*,V^*)=0\ a.e.$ 
over $[y^*,y]$ for any $y \in [y^*,y_0]$
and using $\sF_3(\lambda^*,\rho_i^*,R_e^*,V^*)=0$, we have
\begin{equation}\lb{es1}
(\partial_x V^*+\lambda^*)(\partial_x^2 V^*+e^{-\frac{\lambda^*}{2}y}R_{e}^*)(y)
\leq \int_{y^*}^y \frac{k_e}{k_i}h(\partial_x V^*+\lambda^*)e^{-\frac{\lambda^*}{2}z}R_{e}^*\,dz
\quad \text{for a.e. $y \in [y^*,y_0]$}. 
\end{equation}
By \er{p1} and \er{singular1}, the left hand side is estimated from below as
\begin{align*}
(\partial_x V^*+\lambda^*)(\partial_x^2 V^*+e^{-\frac{\lambda^*}{2}y}R_{e}^*)
\geq (\partial_x V^*+\lambda^*)\partial_x^2 V^*
=\frac{1}{2}\partial_x \left\{\left(\partial_x V^*+\lambda^*\right)^2\right\} \ a.e. 
\end{align*}
since $\partial_x V^*$ is absolutely continuous.  
The integrand on the right hand side of \eqref{es1} is estimated from above by 
$C e^{-b(\partial_x V^*+\lambda^*)^{-1}}$, due to the behavior of $h$; see \eqref{gh}.  
Consequently, substituting these expressions into \er{es1}, 
integrating the result over $[y^*,x]$, 
and using $(\partial_x V^*+\lambda^*)(y^*)=0$, we have
\begin{equation}\lb{es2}
\left(\partial_x V^*+\lambda^*\right)^2(x)
\leq C \int_{y^*}^x\int_{y^*}^y e^{-b(\partial_x V^*(z)+\lambda^*)^{-1}} \,dzdy
\quad \text{for $x \in [y^*,y_0]$}. 
\end{equation}
Now let us define $x_n$ by
\[
 x_n:=\inf\left\{x\leq y_0;\ \ \partial_x V^*(x)+\lambda^* = \frac{1}{n}\right\}.
\]
Notice that $y^* < x_n$ and $(\partial_x V^*+\lambda^*)(x)\leq 1/n$ for any $x\in [y^*,x_n]$,
since the continuous function $(\partial_x V^*+\lambda^*)$ vanishes at $x=y^*$.
Then we evaluate \er{es2} at $x=x_n$ to obtain
\[
\frac{1}{n^2}
\leq C \int_{y^*}^{x_n}\int_{y^*}^y e^{-b(\partial_x V^*(z)+\lambda^*)^{-1}} \,dzdy
\leq C e^{-{bn}}.
\]
For suitably large $n$, this clearly does not hold.
So once again we have  a contradiction.

The remaining case is that $x_*=0$ and $\partial_x V^*+\lambda^*\equiv 0$ .
In this case,  $\partial_x^2 V^*\equiv 0$ and so 
the equation  $\sF_2(\lambda^*,\rho_i^*,R_e^*,V^*)  =0$ yields
$\partial_x^2(e^{-\la^* x/2}R_e^*)  
=  e^{-\la^* x/2}(\partial_x^2R_e^* - \la\partial_x R_e^* + \frac{\la^2}{4} R_e^*) = 0$.
This means that $e^{-\la^* x/2}R_e^*(x)=cx+d$ for some constants $c$ and $d$.
Furthermore, $d=0$ also follows from \er{boundary1}. 
On the other hand, \er{bc0} holds for any $s_n>0$ and 
then using \er{converge3} and $(\partial_x V^*+\lambda^*)\equiv 0$, we have
\begin{align*}
\partial_xR_e^*(L)=-\left(\partial_x V^*(L)+\frac{\la^*}{2}\right)R_e^*(L)
+\gamma e^{\frac{\la^*}{2}L}
\int_0^L h\left(\partial_x V^*(x)+\lambda^*\right)e^{-\frac{\lambda^*}{2}x}R_e^*(x) \,dx
=\frac{\la^*}{2}R_e^*(L).
\end{align*}
Substituting $R_e^*(x)=cxe^{\la^* x/2}$, we find $c=0$.  
Consequently, $R_e^* \equiv 0$.  
Then  
we obtain $\rho_i^* \equiv 0$ from $\sF_3(\lambda^*,\rho_i^*,R_e^*,V^*)=0$ .
Solving $\sF_3(\lambda^*,\rho_i^*,R_e^*,V^*)=0$ with \er{boundary1} 
and $\rho_i^* \equiv R_e^* \equiv 0$,
we also have $V^* \equiv 0$.
Consequently $\la^*=0$ holds and  $\varliminf_{s\to0}\la(s)=0$.
This contradicts Lemma \ref{lem(a)}, since
$\sup_{s>0}\|V(s)\|_{C^2}$ is bounded.
\end{proof}

Next, we reduce Condition (c) in Theorem \ref{Global1} to a simpler condition.
We write the result directly in terms of the ion density $\rho_i$ and
the electron density $\rho_e=R_e e^{-\la x/2}$.

\begin{lem}\lb{lem(d)}
Assume the global positivity alternative (i) in Theorem \ref{Global2}. \  
If $\sup_{s>0}\{\|\rho_i(s)\|_{C^0}+\|\rho_e(s)\|_{C^0}+\lambda(s)\}$ 
is bounded, 
then $\sup_{s>0}\{\|\rho_i(s)\|_{C^1}+\|R_e(s)\|_{C^2}+\|V(s)\|_{C^3}\}$ 
is bounded.
\end{lem}
\begin{proof}
It is clear from $\sF_3=0$
together with the definition  $\rho_e=R_ee^{-\la x/2}$, 
that \begin{equation*}
\sup_{s>0}\|V(s)\|_{C^2}
\leq C \sup_{s>0}\{\|\rho_i(s)\|_{C^0}+\|\rho_e(s)\|_{C^0}\}
<+\infty. 
\end{equation*}
From this, the equation $\sF_2=0$, and $\sup_{s>0}\la(s)<+\infty$, 
we also deduce that $\sup_{s>0}\|R_e(s)\|_{C^2}<+\infty$.
Now  Lemma \ref{lem(b)} implies that  
$\varliminf_{s\to0}\{\inf_{x}(\partial_x V+\lambda)(s,x)\}\}\neq 0$.
Together with  \er{R_i1}, this result  leads to $\sup_{s>0}\|\rho_i(s)\|_{C^1}<+\infty$.
Finally the bound $\sup_{s>0}\|\partial_x^3 V(s)\|_{C^0}$ $<+\infty$
follows from $\sF_3(\lambda(s),\rho_i(s),R_e(s),V(s))=0$.
\end{proof}

We conclude with the following {\it main result}. 

\begin{thm} \label{mainthm}
Assume that the sparking voltage exists (that is, $D$ vanishes somewhere),
and the transversality condition \er{transversality2} holds. 
Then one of the following two alternatives occurs:
\begin{enumerate}[(A)]
\item Both $\rho_i(s,x)$ and $\rho_e(s,x)=(R_ee^{-\la \cdot/2})(s,x)$ 
are positive  for any $s\in(0,\infty)$ and $x \in I$.   Furthermore, 
$\varlimsup_{s\to\infty}\{\|\rho_i(s)\|_{C^0}+\|\rho_e(s)\|_{C^0}+\lambda(s)\}=\infty$.
\item
there exists a finite s-value $s^\ddagger>0$ and a voltage $V_c^\ddagger>V_c^\dagger$ such that 
\begin{enumerate}[(1)]  
\item $D(V^\ddagger)=0,\, g(V_c^\ddagger) \le \pi^2/L^2$ ; 
\item $\rho_i(s,x)>0$ and $\rho_e(s,x)>0$ \ for all $s\in(0,s^\ddagger)$ and $x \in (0,L]$;
\item $(\la(s^\ddagger),\rho_i(s^\ddagger),R_e(s^\ddagger),V(s^\ddagger))=(V_{c}^\ddagger/L,0,0,0)$;
\item  $\rho_i(s,x)<0$ and $\rho_e(s,x)<0$ \ for $0<s-s^\ddagger\ll1$ and $x \in (0,L]$.
\end{enumerate}
\end{enumerate}
\end{thm}

\begin{proof}  Suppose that (B), 
which  is the same as the second alternative $(ii)$ in Theorem \ref{Global2},  does {\it not} hold.  
We will prove (A).  
Then the first alternative $(i)$ in Theorem \ref{Global2} must hold.  
Now in Theorem \ref{Global1} there are four alternatives.  Alternative (d) cannot happen because 
$\rho_i$ and $R_e$ are negative on part of the loop. 
Lemmas \ref{lem(a)} and \ref{lem(b)} assert that either (a) or (b) implies that 
$\sup_{s>0}\{\|\rho_i(s)\|_{C^0}+\|R_e(s)\|_{C^2}+\|V(s)\|_{C^2}+\lambda(s)\}$ is unbounded.
Then Lemma \ref{lem(d)} implies that $\sup_{s>0}\{\|\rho_i(s)\|_{C^0}+\|\rho_e(s)\|_{C^0}+\lambda(s)\}$ 
must also be unbounded.  This means that (A) holds. 
\end{proof} 
This concludes the proof of Theorem \ref{mainthm0}.  
{\  We remark that (B) never occurs 
unless a voltage $V_c^\ddagger > V_c^\dagger$ exists satisfying \er{positive1}. }


\section{Bounded Densities}

It is of interest to know how the global bifurcation curve behaves 
for the case that the densities are bounded but $\lambda$ is unbounded.
We see from \er{sp0} that $(\rho_i,\rho_e,V)$ solves
\begin{subequations}\lb{r5}
\begin{gather}
\partial_x\left\{\left(\partial_x V+\lambda\right)\rho_i\right\}
=\frac{k_e}{k_i}h\left(\partial_x V+\lambda\right)\rho_e,
\lb{re6} \\
-\partial_x\{\left(\partial_x V+\lambda\right)\rho_e +\partial_x\rho_e\}
=h\left(\partial_x V+\lambda\right)\rho_e,
\lb{re7} \\
\partial_x^2 V=\rho_i-\rho_e,
\lb{re8} \\ 
\left(\partial_x V(L)+\lambda\right)\rho_e(L) +\partial_x\rho_e(L)=
\gamma\frac{k_i}{k_e}\left(\partial_x V(L)+\lambda\right)\rho_i(L)
\lb{re9}
\end{gather}
with boundary conditions
\begin{equation}\lb{rb5}
\rho_i(0)=\rho_e(0)=V(0)=V(L)=0.
\end{equation}
\end{subequations}

\begin{lem}\lb{lem(c)}

Assume $\gamma(1+\gamma)^{-1}\neq e^{-aL}$ and that there is a sparking voltage 
\footnote 
{Lemma \ref{A2} ensures that we can have a sparking voltage $V_c^\dagger$ under the 
inequality $\gamma(1+\gamma)^{-1}> e^{-aL}$.}.
Also assume alternative (A) in Theorem \ref{mainthm}.
Furthermore, suppose that there exists a sequence $\{s_n\}_{n \in \mathbb N}$ such that
\begin{equation}\lb{cond5}
\lim_{n\to\infty}s_n=\infty, \quad
\sup_{n \geq 1}(\|\rho_i(s_n)\|_{C^0}+\|\rho_e(s_n)\|_{C^0})<+\infty, \quad
\lim_{n\to\infty}\la(s_n)=\infty.
\end{equation}
Then 
$\lim_{{n}\to\infty} (\|\rho_i(s_{n})\|_{C^0}+\|\rho_e(s_{n})\|_{L^1})=0$. 
\end{lem}
\begin{proof}
First, it is clear from \er{re8} and \er{rb5} that
\begin{equation}\lb{bound_V}
\sup_{n\geq1}\|V(s_n)\|_{C^2}
\leq C \sup_{n\geq1}\{\|\rho_i(s_n)\|_{C^0}+\|\rho_e(s_n)\|_{C^0}\}
<+\infty. 
\end{equation}
Solve \er{re6} for $\partial_x\rho_i$ and write $h$ explicitly from \er{gh}  to obtain
\[
\partial_x\rho_i
=\frac{k_e}{k_i} a \exp\left(\frac{-b}{|\partial_x V+\lambda|}\right) 
\frac{|\partial_x V+\lambda|}{\partial_x V+\lambda}\rho_e
- \frac{\partial_x^2 V}{\partial_x V+\lambda}\rho_i.
\]
From this, \er{cond5}, and \er{bound_V}, we see that
$\sup_{n\geq1}\|\rho_i(s_n)\|_{C^1}<+\infty$
and thus there exist a subsequence [still denoted by $s_n$]
and $(\rho_i^*,\rho_e^*,V^*)$ such that
\begin{gather}
\left\{
\begin{array}{llllll}
 \lambda(s_n) & \to & \infty & \text{in} & \mathbb R,
 \\
 \rho_{i}(s_n) & \to & \rho_i^* & \text{in} & C^0([0,L]), & 
 \\
 \rho_{e}(s_n) & \rightharpoonup & \rho_{e}^* & \text{in} & L^\infty(0,L) & \text{weakly-star},
 \\
 V(s_n) & \to & V^* & \text{in} & C^1([0,L]),
 \end{array}\right.
\lb{converge5}\\
\rho_{i}^*(0)=V^*(0)=V^*(L)=0,
\lb{boundary5} \\
\rho_i^* \geq 0, \quad \rho_e^* \geq 0.
\lb{p5}
\end{gather}

For the completion of the proof, we claim that it suffices to prove the identity 
\begin{equation}\lb{claim*}
-a\gamma \int_0^L \rho_e^*(y)\,dy + \rho_e^*(x)= a \int_x^L \rho_e^*(y)\,dy \quad a.e. 
\end{equation}
In order to prove this claim, first note that \er{claim*} implies that $\rho_e^*$ is a continuous function. 
Now multiplying the identity by $e^{ax}$, we have
\begin{equation*}
\partial_x\left(e^{ax} \int_x^L \rho_e^*(y)\,dy\right)=-a\gamma \int_0^L \rho_e^*(y)\,dy \,e^{ax}
\quad a.e. 
\end{equation*}
Then integration  over $[0,L]$ leads to
\begin{equation*}
\int_0^L \rho_e^*(y)\,dy \, \{1-\gamma(e^{aL}-1)\}=0,
\end{equation*}
which together with the assumption $\gamma(1+\gamma)^{-1}\neq e^{-aL}$ 
means that $\|\rho_e^*\|_{L^1}=0$.  
We also see from \er{converge5} and $\rho_e(s_n)\geq 0$ that
\begin{equation}\lb{converge6}
\|\rho_e(s_n)\|_{L^1}=\int_0^L 1\cdot \rho_e(s_n,x)\,dx \to 
\int_0^L 1\cdot \rho_e^*(x)\,dx=0 \quad \text{as $n\to\infty$}. 
\end{equation}
It follows that $\|\rho_e(s_n)\|_{L^1} \to 0$  {\it for the whole original sequence}.
Furthermore, solving \er{re6} with \er{boundary5}, we have
\begin{align}   \label{Kny} 
\rho_i(s_n,x)
=a\frac{k_e}{k_i}\int_0^x   K_n(y)   \rho_e(s_n,y) \,dy
\leq C \|\rho_e(s_n)\|_{L^1},
\end{align}				where 
\[ 
K_n(y)   :=  \exp\left(\frac{-b}{|\partial_x V(s_n,y)+\lambda(s_n)|}\right)
\frac{|\partial_x V(s_n,y)+\lambda(s_n)|}{\partial_x V(s_n,x)+\lambda(s_n)}.
\]
Here we have used \er{bound_V} in derving the last inequality.
Together with \er{converge6} this completes the proof of the lemma.

It remains to prove \er{claim*}.
Integrating \er{re7} over $[x,L]$, using \er{re9}, 
and multiplying the result by $\lambda^{-1}$, we obtain
\begin{align*}
&-\gamma\frac{k_i}{k_e}\left(\frac{\partial_x V(L)}{\lambda}+1\right)\rho_i(L)
+\left(\frac{\partial_x V(x)}{\lambda}+1\right)\rho_e(x) 
+\frac{1}{\lambda}\partial_x\rho_e(x)
\\
&=a\int^L_x \exp\left(\frac{-b}{|\partial_x V(y)+\lambda|}\right)
\left|\frac{\partial_x V(y)}{\lambda}+1\right| \rho_e(y)\,dy.
\end{align*}
We take this identity at $s=s_n$ and look at the behavior of each term as $s_n\to\infty$.  
We multiply it by a test function $\phi \in C_c^\infty((0,L))$, 
integrate it over $(0,L)$, and let $n \to \infty$.
Then we notice from \er{cond5}--\er{converge5} that 
\begin{align*}
-\gamma\frac{k_i}{k_e}\left(\frac{\partial_x V(s_n,L)}{\lambda(s_n)}+1\right)\rho_i(s_n,L)\int_0^L   \phi(x) \,dx
&\to -\gamma\frac{k_i}{k_e}\rho_i^*(L)\int_0^L  \phi(x) \,dx,
\\
\int_0^L \left(\frac{\partial_x V(s_n,x)}{\lambda(s_n)}+1\right)\rho_e(s_n,x)  \phi(x) \,dx
&\to \int_0^L \rho_e^*(x) \phi(x) \,dx,
\\
\int_0^L \frac{1}{\lambda}\partial_x\rho_e(s_n,x) \phi(x) \,dx
=-\int_0^L \frac{1}{\lambda}\rho_e(s_n,x) \partial_x\phi(x) \,dx
&\to 0.
\end{align*}
Furthermore, there holds that
\begin{align*}
\int_0^L \left[a\int^L_x \exp\left(\frac{-b}{|\partial_x V(s_n,y)+\lambda(s_n)|}\right)
\left|\frac{\partial_x V(s_n,y)}{\lambda(s_n)}+1\right| \rho_e(s_n,y)\,dy \right] \phi(x)\,dx
=I_{1,n}+I_{2,n},
\end{align*}
where 
\begin{align*}
I_{1,n}&:=\int_0^L \left[a\int^L_x \left\{\exp\left(\frac{-b}{|\partial_x V(s_n,y)+\lambda(s_n)|}\right)
\left|\frac{\partial_x V(s_n,y)}{\lambda(s_n)}+1\right|-1\right\} \rho_e(s_n,y)\,dy \right] \phi(x)\,dx,
\\
I_{2,n}&:=\int_0^L \left[ a \int^L_x \rho_e(s_n,y)\,dy \right]\phi(x)\,dx.
\end{align*}
Then it is also seen from \er{cond5}--\er{converge5} that
\begin{align*}
|I_{1,n}| &\leq       \|\rho_e\|_{L^1}\|\phi\|_{L^1} 
 \sup_{y \in [0,L]}  \left|\exp\left(\frac{-b}{|\partial_x V(s_n,y)+\lambda(s_n)|}\right)
\left|\frac{\partial_x V(s_n,y)}{\lambda(s_n)}+1\right|-1\right| \to 0,
\\
I_{2,n} & = a \int_0^L \rho_e(s_n,y) \left[\int_0^y \phi(x) \,dx \right]\,dy
\\
&\to a \int_0^L \rho_e^*(y) \left[\int_0^y \phi(x) \,dx \right]\,dy
=\int_0^L \left[ a \int^L_x \rho_e^*(y)\,dy \right]\phi(x)\,dx.
\end{align*}
Therefore, we conclude that
\begin{equation}\lb{rhoe*}
-\gamma\frac{k_i}{k_e}\rho_i^*(L) + \rho_e^*(x)=  a \int_x^L \rho_e^*(y)\,dy \quad a.e. 
\end{equation}

Comparing with \er{claim*}, it  is left  to show that
\begin{equation}\lb{rhoi*}
\rho_i^*(L)=a\frac{k_e}{k_i}\int_0^L \rho_e^*(y)\,dy. 
\end{equation}
Indeed, plugging \er{rhoi*} into \er{rhoe*} leads to \er{claim*}.
Evaluating \er{Kny} at $x=L$, we have 
\begin{align*}
\rho_i(s_n,L)
&=a\frac{k_e}{k_i}\int_0^L   K_n(y)   \rho_e(s_n,y) \,dy .  
\end{align*}
Now $K_n\to 1$ uniformly and $\rho_e(s_n) \rightharpoonup \rho_e^*$ in $L^\infty$ weakly-star.  
Therefore, letting $n\to\infty$, 
we get \er{rhoi*} in the limit.  
\end{proof}
}


\begin{appendix}
\section{Roots of the Sparking Function $D$}\lb{SA}
In this appendix we investigate the roots of $D(V_c)$.
The first lemma means that in the case of Figure \ref{fig3}, 
which we discussed in our first paper \cite{SS1},
$D$ always has at least one root.

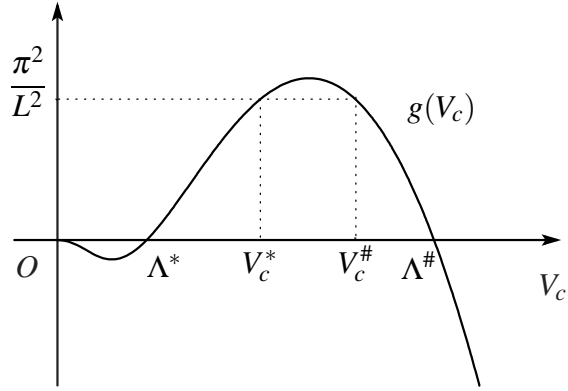
\begin{figure}[htbp]  \begin{center}
 \begin{minipage}{0.5\hsize}   
{\unitlength 0.1in
\begin{picture}( 29.1500, 19.9000)( -0.1100,-22.0000)
\put(1.8900,-15.2100){\makebox(0,0)[rt]{$O$}}%
\put(28.8300,-16.8600){\makebox(0,0){$V_c$}}%
%
{{%
\special{pn 13}%
\special{pa 288 2200}%
\special{pa 288 210}%
\special{fp}%
\special{sh 1}%
\special{pa 288 210}%
\special{pa 268 278}%
\special{pa 288 264}%
\special{pa 308 278}%
\special{pa 288 210}%
\special{fp}%
}}%
%
{{%
\special{pn 13}%
\special{pa 62 1438}%
\special{pa 2904 1438}%
\special{fp}%
\special{sh 1}%
\special{pa 2904 1438}%
\special{pa 2836 1418}%
\special{pa 2850 1438}%
\special{pa 2836 1458}%
\special{pa 2904 1438}%
\special{fp}%
}}%
{{%
\special{pn 13}%
\special{pn 13}%
\special{pa 290 2201}%
\special{pa 290 2187}%
\special{ip}%
\special{pa 290 2153}%
\special{pa 290 2140}%
\special{ip}%
\special{pa 290 2106}%
\special{pa 290 2092}%
\special{ip}%
\special{pa 290 2058}%
\special{pa 290 2044}%
\special{ip}%
\special{pa 290 2010}%
\special{pa 290 1997}%
\special{ip}%
\special{pa 290 1963}%
\special{pa 290 1949}%
\special{ip}%
\special{pa 290 1915}%
\special{pa 290 1901}%
\special{ip}%
\special{pa 290 1868}%
\special{pa 290 1854}%
\special{ip}%
\special{pa 290 1820}%
\special{pa 290 1806}%
\special{ip}%
\special{pa 290 1772}%
\special{pa 290 1759}%
\special{ip}%
\special{pa 290 1725}%
\special{pa 290 1711}%
\special{ip}%
\special{pa 290 1677}%
\special{pa 290 1663}%
\special{ip}%
\special{pa 290 1630}%
\special{pa 290 1616}%
\special{ip}%
\special{pa 290 1582}%
\special{pa 290 1568}%
\special{ip}%
\special{pa 290 1534}%
\special{pa 290 1520}%
\special{ip}%
\special{pa 290 1487}%
\special{pa 290 1473}%
\special{ip}%
\special{ip}%
\special{pa 290 1439}%
\special{pa 306 1439}%
\special{pa 310 1440}%
\special{pa 320 1441}%
\special{pa 326 1442}%
\special{pa 336 1444}%
\special{pa 340 1445}%
\special{pa 350 1448}%
\special{pa 390 1464}%
\special{pa 406 1472}%
\special{pa 420 1480}%
\special{pa 426 1483}%
\special{pa 466 1506}%
\special{pa 470 1509}%
\special{pa 490 1519}%
\special{pa 520 1531}%
\special{pa 526 1533}%
\special{pa 530 1534}%
\special{pa 550 1538}%
\special{pa 556 1539}%
\special{pa 586 1539}%
\special{pa 596 1538}%
\special{pa 600 1537}%
\special{pa 606 1536}%
\special{pa 610 1535}%
\special{pa 616 1534}%
\special{pa 620 1532}%
\special{pa 626 1530}%
\special{pa 646 1522}%
\special{pa 656 1517}%
\special{pa 660 1514}%
\special{pa 676 1505}%
\special{pa 680 1502}%
\special{pa 686 1498}%
\special{pa 690 1495}%
\special{pa 696 1491}%
\special{pa 716 1475}%
\special{pa 720 1471}%
\special{pa 726 1466}%
\special{pa 730 1462}%
\special{pa 736 1457}%
\special{pa 766 1427}%
\special{pa 770 1422}%
\special{pa 776 1416}%
\special{pa 780 1411}%
\special{pa 786 1405}%
\special{pa 836 1345}%
\special{pa 840 1339}%
\special{pa 846 1332}%
\special{pa 850 1326}%
\special{pa 856 1319}%
\special{pa 860 1313}%
\special{pa 866 1306}%
\special{pa 876 1293}%
\special{pa 880 1286}%
\special{pa 896 1266}%
\special{pa 900 1259}%
\special{pa 960 1175}%
\special{pa 966 1168}%
\special{pa 1016 1097}%
\special{pa 1020 1090}%
\special{pa 1060 1034}%
\special{pa 1066 1027}%
\special{pa 1096 986}%
\special{pa 1100 979}%
\special{pa 1110 966}%
\special{pa 1116 959}%
\special{pa 1126 946}%
\special{pa 1130 939}%
\special{pa 1136 933}%
\special{pa 1146 920}%
\special{pa 1150 914}%
\special{pa 1156 907}%
\special{pa 1160 901}%
\special{pa 1220 829}%
\special{pa 1226 823}%
\special{pa 1236 812}%
\special{pa 1240 806}%
\special{pa 1246 801}%
\special{pa 1256 790}%
\special{pa 1260 785}%
\special{pa 1306 740}%
\special{pa 1310 735}%
\special{pa 1320 726}%
\special{pa 1326 721}%
\special{pa 1330 717}%
\special{pa 1346 704}%
\special{pa 1350 700}%
\special{pa 1370 684}%
\special{pa 1376 680}%
\special{pa 1386 673}%
\special{pa 1390 669}%
\special{pa 1396 666}%
\special{pa 1400 662}%
\special{pa 1406 659}%
\special{pa 1446 635}%
\special{pa 1450 632}%
\special{pa 1470 622}%
\special{pa 1476 620}%
\special{pa 1506 608}%
\special{pa 1510 606}%
\special{pa 1516 605}%
\special{pa 1520 603}%
\special{pa 1526 602}%
\special{pa 1536 599}%
\special{pa 1540 598}%
\special{pa 1560 594}%
\special{pa 1566 593}%
\special{pa 1586 591}%
\special{pa 1626 591}%
\special{pa 1630 592}%
\special{pa 1636 592}%
\special{pa 1640 593}%
\special{pa 1650 594}%
\special{pa 1656 595}%
\special{pa 1666 597}%
\special{pa 1670 598}%
\special{pa 1680 601}%
\special{pa 1686 602}%
\special{pa 1690 604}%
\special{pa 1696 605}%
\special{pa 1700 607}%
\special{pa 1726 617}%
\special{pa 1756 632}%
\special{pa 1760 635}%
\special{pa 1766 638}%
\special{pa 1770 641}%
\special{pa 1786 651}%
\special{pa 1790 654}%
\special{pa 1796 658}%
\special{pa 1836 690}%
\special{pa 1840 694}%
\special{pa 1846 699}%
\special{pa 1886 739}%
\special{pa 1890 744}%
\special{pa 1896 750}%
\special{pa 1930 792}%
\special{pa 1936 798}%
\special{pa 1940 805}%
\special{pa 1950 818}%
\special{pa 1956 825}%
\special{pa 1970 846}%
\special{pa 1976 853}%
\special{pa 1996 883}%
\special{pa 2000 891}%
\special{pa 2020 923}%
\special{pa 2026 931}%
\special{pa 2030 940}%
\special{pa 2036 948}%
\special{pa 2040 957}%
\special{pa 2076 1020}%
\special{pa 2080 1029}%
\special{pa 2086 1039}%
\special{pa 2120 1109}%
\special{pa 2126 1119}%
\special{pa 2130 1130}%
\special{pa 2170 1218}%
\special{pa 2180 1241}%
\special{pa 2186 1253}%
\special{pa 2206 1301}%
\special{pa 2210 1313}%
\special{pa 2216 1326}%
\special{pa 2220 1338}%
\special{pa 2226 1351}%
\special{pa 2260 1442}%
\special{pa 2270 1469}%
\special{pa 2276 1483}%
\special{pa 2296 1539}%
\special{pa 2300 1553}%
\special{pa 2310 1582}%
\special{pa 2350 1702}%
\special{pa 2356 1717}%
\special{pa 2360 1733}%
\special{pa 2386 1813}%
\special{pa 2390 1829}%
\special{pa 2400 1862}%
\special{pa 2436 1981}%
\special{pa 2440 1998}%
\special{pa 2446 2016}%
\special{pa 2450 2033}%
\special{pa 2456 2051}%
\special{pa 2470 2105}%
\special{pa 2476 2123}%
\special{pa 2496 2197}%
\special{pa 2496 2201}%
\special{fp}%
}}%
\put(23.0000,-7.5000){\makebox(0,0){$g(V_c)$}}%
%
{{%
\special{pn 8}%
\special{pa 1350 700}%
\special{pa 1350 1438}%
\special{dt 0.045}%
\special{pa 1350 1438}%
\special{pa 1338 1438}%
\special{dt 0.045}%
}}%
\put(13.5000,-15.7000){\makebox(0,0){$V_c^*$}}%
\put(0.3000,-7.9000){\makebox(0,0)[lb]{$\displaystyle \frac{\pi^2}{L^2}$}}%
%
{{%
\special{pn 8}%
\special{pa 1850 690}%
\special{pa 1850 1428}%
\special{dt 0.045}%
\special{pa 1850 1428}%
\special{pa 1838 1428}%
\special{dt 0.045}%
}}%
\put(18.5000,-15.6000){\makebox(0,0){$V_c^\#$}}%
%
{{%
\special{pn 8}%
\special{pa 1850 700}%
\special{pa 280 700}%
\special{dt 0.045}%
\special{pa 280 700}%
\special{pa 290 710}%
\special{dt 0.045}%
}}%
\put(8.5000,-15.7000){\makebox(0,0){$\Lambda^*$}}%
\put(21.8000,-15.7000){\makebox(0,0){$\Lambda^{\#}$}}%
\end{picture}}%
  \caption{\small local max is greater than ${\pi^2}/L^2$}  \label{fig3} \end{minipage}\end{center}\end{figure}

\begin{lem}
\lb{A1}
 (i)  If $\max_{V_c>0}g(V_c)>{\pi^2}/L^2$, then 
$D(V_c)$ has at least one root $V_c$ that satisfies $g(V_c^\dagger)<\pi^2/L^2$  in the interval $(0,V_c^*)$.
(ii)  In addition, if 
\begin{equation}\lb{CondA1}
a> 4^{-1}eb+ e\pi^2b^{-1}. 
\end{equation}
then 
$\max_{V_c>0}g(V_c)>{\pi^2}/L^2$. 
\end{lem}
\begin{proof}  
(i)  Since $\max g>0$, the function $g$ has exactly two positive roots.  
Define $V_c^*$ by 
\begin{equation}\label{V^*}
g(V_c^*)=\frac{\pi^2}{L^2}, \quad g'(V_c^*)>0  
\end{equation}
as in Figure \ref{fig3}.
We have 
\[
D(V_c^*)=-1-\frac{\gamma}{1+\gamma}e^{V_c^*/2}<0. 
\]
In addition, $\lim_{V_c\to 0}D(V_c)=\frac{1}{1+\gamma}>0$.  
So we see that $D(V_c)$ has at least one root $V_c$ that satisfies \er{positive1} 
on the interval $(0,V_c^*)$.
For (ii)  we simply note that \eqref{A1} implies that $g(b)>\pi^2/L^2$.  
\end{proof}

We also can find a sufficient condition 
for the existence of roots of $D$ that is caused by the $\gamma$-mechanism.
In this case it does not matter whether or not $\max_{V_c>0}g(V_c)>{\pi^2}/L^2$ holds.

\begin{lem}
\lb{A2}
Suppose that
\begin{equation}\lb{CondA2}
\gamma(1+\gamma)^{-1}>e^{-aL}.
\end{equation}
Then $D(V_c)$ has at least one root.
\end{lem}
\begin{proof}
First $\lim_{V_c\to 0}D(V_c)=\frac{1}{1+\gamma}>0$ holds.
We also see that  
$\mu = L\sqrt{g(-V_c)} = V_c/2 - aL + O(V_c^{-1})$ as $V_c\to\infty$.  Thus 
\[
\lim_{V_c\to \infty} \frac{D(V_c)}{e^{V_c/2}}
=e^{-aL}-\frac{\gamma}{1+\gamma} <0,
\]
which means $\lim_{V_c\to \infty}D(V_c)=-\infty$.
Hence $D$ has a positive root.
\end{proof}

We remark the roots in Lemmas \ref{A1} and \ref{A2} are sparking voltages.
Indeed for a fixed triple $(a,b,\gamma)$ in the open set
$\{(a,b,\gamma)\in(\real_+)^3 \ ; \ \text{either \eqref{CondA1} or \eqref{CondA2} holds} \}$,
the sparking function $D$ has a positive root for any triple in a neighborhood of it.

In the next lemma,
we find a candidate of the {\it anti-sparking voltage} $V_c^{\ddagger}$.
Therefore alternative (B) in Theorem \ref{mainthm} is an actual possibility.

\begin{lem}\lb{A5}
Let \er{CondA1} hold. 
There exists a positive constant $\gamma_0$ such that if $\gamma <\gamma_0$,
then $D(V_c)$ has at least two roots $V_c^{\dagger}$ and $V_c^{\ddagger}$ with \er{positive1}.
\end{lem}

\begin{proof}
We know from Lemma \ref{A1} and its proof that
a root $V_c$ with \er{positive1} exists in the open interval $(0,V_c^*)$.
Let us seek another root $V_c^{\ddagger}$.
The graph of $g$ is sketched in Figure \ref{fig3} and thus $V_c^{\#}$
is the unique value such that  
$g(V_c^\#)=\frac{\pi^2}{L^2}$ and $g'(V_c^\#)<0$.
The function $g$ has three roots $0$, $\La^*$, and $\La^{\#}$ such that  
$0 < \La^* < V_c^* < V_c^{\#} < \La^{\#}$.
We emphasize that $g$ is independent of $\gamma$. 
Now consider the function $D$ on the interval $[V_c^\#,\La^\#]$.  
Evaluating $D(V_c)$ at the point $V_c=V_c^{\#}$, 
we have 
\[
D(V_c^{\#})=-1-\frac{\gamma}{1+\gamma}e^{V_c^{\#}/2}<0. 
\]
On the other hand, evaluating $D(V_c)$ at the point $V_c=\La^{\#}$, we have
\[
D(\La^{\#})=1+\frac{\La^{\#}}{2}-\frac{\gamma}{1+\gamma}e^{\La^{\#}/2}>0,
\]
where the last inequality is valid for suitably small $\gamma$.
Thus there must be a root in between;   
that is,  $D$ has a root with \er{positive1} in the open interval $(V_c^{\#},\La^{\#})$.
\end{proof}

The next lemma states a sufficient condition for the absence of any root of $D$.

\begin{lem}
\lb{A3}
If $a< 4^{-1}eb$ and $\gamma(1+\gamma)^{-1}\leq e^{-2aL}$ hold, 
then $D(V_c)$ has no root.
\end{lem}
\begin{proof}
We first claim that $g(\lambda L)$ is negative for any $\lambda=V_c/L \in (0,\infty)$ 
if and only if $a< 4^{-1}eb$.
Indeed, $g(\lambda L) < 0$ holds 
if and only if $a e^{-b/\lambda}< \frac{\lambda}{4}$ holds.
By taking logarithms, we see that $g(\lambda L) < 0$ is equivalent to  
$G(\lambda):=-\log\lambda -\frac{b}{\lambda}+\log a+\log 4 < 0$.
It is straightforward to check that $G$ attains a maximum at $\lambda=b$.
Furthermore, the maximum is less than zero if and only if $a< 4^{-1}eb$.
This proves the claim.

For any $\la=V_c/L >0$, the negativity of $g(\la L)$  implies that 
\begin{align}
{D(\lambda L)} &=
\cosh(L\sqrt{-g(\la L)}) + \frac\la{2\sqrt{-g(\la L)}} \sinh(L\sqrt{-g(\la L)})  -\frac\ga{1+\ga} e^{\frac\la2L}
\notag\\ 
&> e^{L\sqrt{-g(\la L)}}   -\frac\ga{1+\ga} e^{\frac\la2L}  
\label{Aeq2}
\end{align}
because 
\[
\frac{\lambda}{2\sqrt{-g(\lambda L)}} 
>  \frac{\lambda}{2\sqrt{\lambda^2/4}}=1   \]
and $\cosh z + \sinh z = e^z$.  
Now we note that 
\[  
\sqrt{-g(\lambda L)}-\tfrac{\lambda}{2}  
=-h(\lambda) \left(\sqrt{-g(\lambda L)}+\tfrac{\lambda}{2} \right)^{-1}  
> -h(\lambda) (\tfrac{\lambda}{2})^{-1}  
>  -2a.  \]
So the right side of \er{Aeq2} is greater than 
\[  
e^{\frac\la2 L }  \{ e^{-2aL} - \tfrac\ga{1+\ga}  \}  \ge 0  \]
by hypothesis.  
\end{proof}

Independently, it can be shown numerically that $D(V_c)$ has a unique root or many roots
for suitable choices of $L$, $a$, $b$, and $\gamma$. 
To illustrate this, the graphs of $(1+\gamma)e^{-V_c/2}D(V_c)$ are sketched in Figures \ref{fig4} and \ref{fig5}
for the case $L=1$, $a=3$, $b=4$, $\gamma=5$ and for the case 
$L=1$, $a=70$, $b=0.1$, $\gamma=0.1$, respectively.
\begin{figure}[htbp]
 \begin{minipage}{0.47\hsize}
   \includegraphics[width=0.99\textwidth]{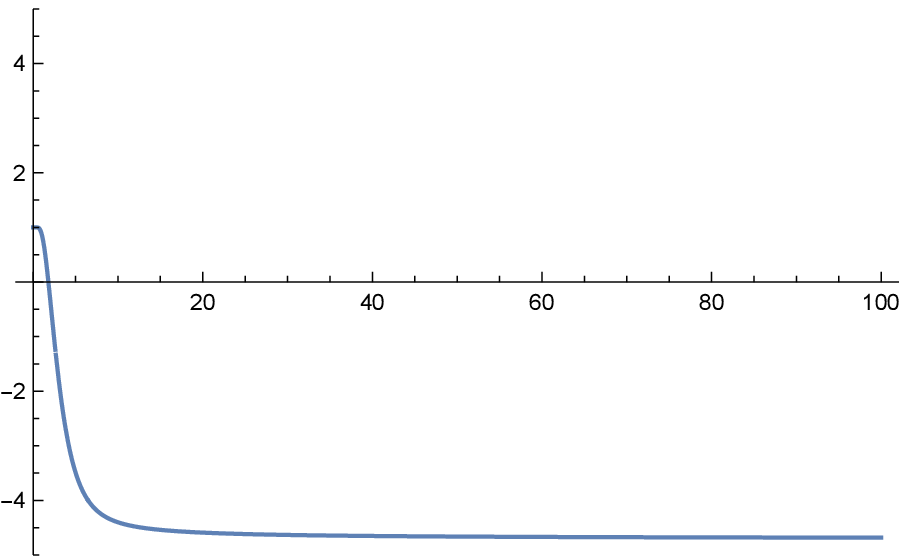}
  \caption{\small unique root}
  \label{fig4}
 \end{minipage} \mbox{}
 \begin{minipage}{0.47\hsize}   
   \includegraphics[width=0.99\textwidth]{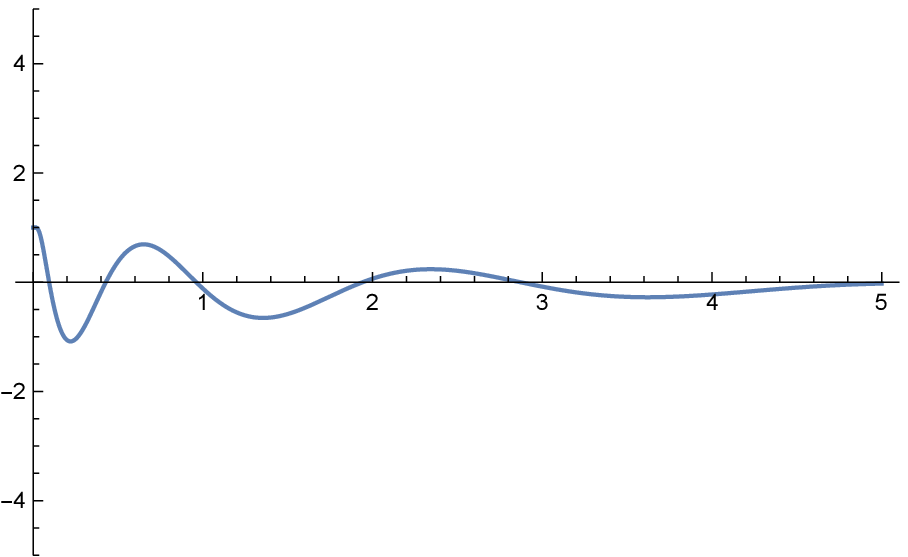}
  \caption{\small many roots}  
  \label{fig5}
 \end{minipage}
\end{figure}


\section{The Sparking Voltage $V_c^\dagger$}\lb{SB}

In this brief appendix we illustrate the location of the sparking voltage 
if $\gamma$ is very small or very large. Let $V_c^*$ be defined in \er{V^*}.

\begin{lem}\lb{B1}
Suppose that $\max_{V_c>0}g(V_c)>{\pi^2}/L^2$ (see Figure \ref{fig3}).  
If $\gamma$ is sufficiently small,
then $V_c^\dagger < V_c^*$ and $\frac{\pi^2}{4L^2} < g(V_c^\dagger) < \frac{\pi^2}{L^2}$.  
\end{lem}
\begin{proof}
We know from Lemma \ref{A1} that $V_c^\dagger < V_c^*$ and $g(V_c^\dagger) < \frac{\pi^2}{L^2}$. 
It only remains to show that $g(V_c^\dagger)>\frac{\pi^2}{4L^2}$.
By continuity it suffices to prove the strict inequalities of the conclusion in case $\gamma=0$.
We begin by proving that $g(V_c^\dagger)>0$.
On the contrary, suppose that $g(V_c^\dagger)\leq 0$.
This assumption and $\gamma=0$ lead to $D(V_c^\dagger)>0$, which contradicts to 
the fact that $V_c^\dagger$ is the sparking voltage, that is, $D(V_c^\dagger)=0$.
Now let us suppose that $0<g(V_c^\dagger)\leq \frac{\pi^2}{4L^2}$.
We see from $D(V_c^\dagger)=0$ that
\begin{equation*}
 \frac{V_c^\dagger}{2L\sqrt{g(V_c^\dagger)}}  =  -\cot \left(L \sqrt{g(V_c^\dagger)}\right).  
\end{equation*}
The signs are contradictory.
Thus we conclude that $g(V_c^\dagger)>\frac{\pi^2}{4L^2}$.
\end{proof}

\begin{lem}\lb{B2}
Suppose that $\max_{V_c>0}g(V_c)>0$ (see Figure \ref{fig1}).
There exists $\Gamma>0$ such that for $\gamma>\Gamma$, 
we have $V_c^\dagger\in(0,\La^*)$, where 
$\La^*$ is the smallest positive root of $g(V_c)=0$.
\end{lem}
\begin{proof}
We first see that $\lim_{V_c\to 0}D(V_c)=\frac{1}{1+\gamma}>0$.
Evaluating $D$ at $V_c=\La^*$ and using $g(\La^*)=0$, we have
\begin{equation*}
D(\La^*)= 1 
+ \frac{\La^*}{2}  
-   \frac{\gamma}{1+\gamma}  e^{\frac{\La^*}{2}}
<\frac{1}{1+\gamma}  \left(1+\frac{\La^*}{2}\right)
-   \frac{1}{2}\frac{\gamma}{1+\gamma}  \left(\frac{\La^*}{2}\right)^2
<0.
\end{equation*} 
In deriving the last inequality, we have taken $\gamma$ suitably large.
Therefore, the intermediate value theorem gives $V_c^\dagger\in(0,\La^*)$.
\end{proof}
\end{appendix}



\end{document}